\documentclass[11pt,reqno]{amsart}
\usepackage[utf8]{inputenc}
\usepackage{pgfplots}
\usepackage{fancyhdr}
\usepackage{tikz}
\usepackage{systeme}
\usepackage{graphicx}
\usepackage[rightcaption]{sidecap}
\graphicspath{ {images/} }
\usepackage[utf8]{inputenc}
\usepackage{amsmath,dsfont,amssymb,mathrsfs,amsthm,hyperref,color}
\usepackage{graphicx}
\usepackage[english]{babel}
 
\usepackage{pst-node}
\usepackage{tikz-cd}

\newtheorem{theorem}{Theorem}[section]
\newtheorem{corollary}[theorem]{Corollary}
\newtheorem{lemma}[theorem]{Lemma}
\newtheorem{proposition}[theorem]{Proposition}
\theoremstyle{definition}
\newtheorem{definition}[theorem]{Definition}

\theoremstyle{remark}    
\newtheorem{remark}[theorem]{Remark}

\newtheorem{Condition}{Condition}

\title{An Optimal Transportation Principle for Interacting Paths and Congestion} 

     \author{Rene Cabrera}
     \email{cabrera@math.umass.edu}
\begin{document}
 \maketitle 
 
\begin{abstract}
In this work we study a modification of the Monge-Kantorovich problem taking into account path dependence and interaction effects between particles. We prove existence of solutions under mild conditions on the data, and after imposing stronger conditions, we characterize the minimizers by relating them to an auxiliary Monge-Kantorovich problem of the more standard kind. With this notion of how particles interact and travel along paths, we produce a dual problem. The main novelty here is to incorporate an interaction effect to the optimal path transport problem.  This covers for instance, $N$-body dynamics when the underlying measures are discrete. Lastly,  our results include an extension of Brenier's theorem  on optimal transport maps.
\end{abstract}

\tableofcontents
\newpage
\section{Introduction}  
In 1781 Gaspard Monge initiated the problem of how to transfer mass from an initial location onto a final location in the most efficient way possible \cite{villani2003topics}.  He interpreted this problem mathematically using Euclidean geometry.  To wit, the cost of transferring a unit of mass from location $x$ to location $y$ was interpreted as the Euclidean distance $c(x,y):=|x-y|$.  But this turned out to be quite difficult to solve. In the mid 20th century, Kantorovich studied a relaxation of this problem, that reduces it to a linear optimization problem \cite{kantorovich1960mathematical}.  In the early 90's, Brenier's work \cite{Brenier1991} gave a new impetus to the field which has considerably expanded and matured in the last three decades. The interested reader can find more about the history of the field, for example in the books by Villani \cite{villani2003topics}, \cite{villani2009optimal} and Santambrogio \cite{santambrogio2015optimal}. As a result, the optimal transportation in its modern formulation (see below) is known as the Monge-Kantorovich problem (MKP). Monge set the precedent to study the   optimal path transport problem to incorporate particle trajectories \cite{villani2003topics}. Benamou and Brenier, however, \cite{benamou2000computational} were the ones who intentionally reintroduced  the time dependent variable to the optimal transport problem in the case of the quadratic cost function. The optimal transport problem may also be viewed as a distance problem between two probability measures, and the time-dependent minimization problem may be viewed as a minimal path problem \cite[Ch 5]{villani2003topics}.

In this paper, we introduce a new variant of the MKP that incorporates interacting paths and congestion. We prove solutions (using the Condition \ref{condition:weaker}) of: \begin{align}\label{MKpath}
    \inf_{\pi\in \Pi_{\text{path}}(\mu_{0}, \mu_{1})}\mathcal{E}_{0}(\pi)+\int_{\Omega}\mathcal{U}(\gamma, \pi)d\pi(\gamma)
\end{align}
exist. The functional $\mathcal{E}_{0}(\pi)$ is the total cost of transporting along all paths in the plan, and this is reminiscent of the MKP but using an explicit path dependence approach. The interacting (latter) term in (\ref{MKpath}) measures the total congestion between paths $\gamma$ and a distribution of paths given by $\pi$. The optimal transport solution, $\pi$, to (\ref{MKpath}) is unique (Theorem \ref{result 4: uniqueness Kantorovich solution with interaction}) and is given by a map, $\Gamma$.   We also formulate a duality (Theorem \ref{theorem: interaction analogue}) to characterize solutions of (\ref{MKpath}). Defining an effective cost (\ref{effectivecost}) with interaction, we prove an extension and a general version of Theorem \ref{result 2: uniqueness Kantorovich solution on paths}; namely, Theorem \ref{result 4: uniqueness Kantorovich solution with interaction} which says that the optimal plan is given by a general optimal map of paths2 characterized by a general result of Brenier, Gangbo and McCann \cite{villani2009optimal}, \cite{Brenier1991}, \cite{gangbo1996geometry}. The central lemmas of the paper that make this possible comes from $c_{e}$-cyclical monotonicity (Definition \ref{definition of cyclical sets}) of optimal transport that incorporates path dependence and interaction effects. Namely Lemmas \ref{lemma on minimal paths} and \ref{c-cyclical monotone}.

Instead of looking at a cost function $c(x,y)$, that represents how much it costs to send a unit of mass at point $x$ to point $y$, we  consider continuous paths $\gamma$ such that $\gamma(0)=x$ indicates the initial point along $\gamma$ and $\gamma(1)=y$ the arrival point along $\gamma$, and associate to such a path 
\[
c(\gamma)=\int_{0}^{1} L(\gamma(t),\dot{\gamma}(t), t)dt,
\]
 indicating \emph{how much transportation along that path costs}. Here, $L$ is a Lagrangian, e.g., $L(\gamma(t),\dot{\gamma}(t), t)=\|\dot{\gamma}(t)\|^{2}-V(\gamma(t),t)$ (more on this in Section \ref{section:minimal paths}). More concretely,  if $\Omega$ denotes the set of such continuous rectifiable paths, then the total cost of transporting along all paths in the plan is the functional:
 \begin{align*}
 \mathcal{E}_{0}(\pi):=\int_{\Omega}c(\gamma)d\pi(\gamma).
\end{align*}
 This can be thought of as an explicit path dependent version of the MKP. Monge set some precedent to this line of reasoning \cite{monge1781memoire}.  Villani \cite{villani2009optimal} explored this line of reasoning by describing a time-dependent version of optimal transport. In this manuscript we explore this further and show it reduces to the traditional optimal transport problem. 
 
 \subsection{Related work}\label{related work} Motivated by the modeling of traffic networks, Carlier, Jimenez, and Santambrogio \cite{carlier2008optimal} studied a related transport problem involving paths, as well. The way \cite{carlier2008optimal} models congestion effects is different from the present paper: in \cite{carlier2008optimal} the objective functional considers intensity through paths while here the functional involves an interaction potential. The novelty in our work comes from introducing an interaction term. Probability measures over the space of paths  are frequently applied in both  probability theory and mathematical physics. For example, Hynd recently considered such measures in the study of 1D sticky particle systems \cite{hynd2018sticky}. Whether this or related PDE models could benefit from the point of view in this paper is an interesting question. At any rate, our optimal path problem with interaction term reduces to Brenier's result and thus solves our optimal path problem  with interaction effect (\ref{MKpath}). 

\subsection{MKP} Let us briefly review the modern description of Monge's problem for the quadratic cost function, $c(x, y)=|x-y|^{2}$. One considers two density functions $f(x)$ and $g(y)$. Then if we have two probability measures $\mu$ on $X\subset\mathbb{R}^{n}$ and $\nu$ on $Y\subset \mathbb{R}^{n}$, then $\mu(x)=f(x)\;dx$ and $\nu(y)=g(y)\;dy$. Suppose $T$ is any (Borel) measurable function of $X\subset\mathbb{R}^{n}$ to  $Y \subset \mathbb{R}^{n}$ such that $T$ \textit{pushes} $\mu$ forward to $\nu$. This is denoted by $T_{\sharp}\mu=\nu$ and it means that for any measurable (Borel) subset
\[
B \subset Y,\; \nu(B):=T_{\sharp}\mu(B)=\mu(T^{-1}(B))\; \text{and where}\; T^{-1}(B):=\{x\in \mathbb{R}^{n}: T(x)\in B\}
\]
Equivalently, $\int_{T^{-1}(B)}f(x)dx=\int_{B}g(y)dy$ for all Borel sets $B$.  When $T$ is measure preserving ($T_{\sharp}\mu=\nu$), using the change of variables formula for any continuous function $h\in C^{0}(\Bar{Y})$,
\begin{align}\label{changevariables}
\int_{Y}h(y)\;d\nu(y)=\int_{X}h(T(x))\;d\mu(x)
\end{align}
is another characterization.\\

\textit{\textbf{Monge's problem}} is: Minimize the total \textit{transportation cost} 
\begin{align}\label{monge min}
\int_{X} c(x,T(x))d\mu(x)=\int_{X}|x-T(x)|^{2}\;d\mu(x)
\end{align}
among all $T$ pushing forward $\mu$ to $\nu$ \cite{AmbGig2013}.

In 1940  Leonid Kantorovich \cite{kantorovich2006translocation} in some sense ``relaxed" Monge's problem \cite{AmbGig2013} and \cite{villani2003topics} by introducing a linear program formulation for the problem. Concretely, again for the quadratic cost function, Kantorovich considered probability measures $\pi$ on $\mathbb{R}^{n}\times \mathbb{R}^{n}$ with left and right marginals $\mu, \nu$, respectively; namely, $\mu[A]=\pi[A\times \mathbb{R}^{n}]$ and $\pi[\mathbb{R}^{n}\times B]=\nu[B]$ for any (Borel) measurable subsets $A,B$ of $\mathbb{R}^{n}$. An equivalent criterion for $\pi$ to have left and right marginals $\mu, \nu$ is the following linearity of $\pi$ \cite{villani2003topics}:
\[
\forall (\varphi, \psi)\in L^{1}(d\mu)\times L^{1}(d\nu),\quad \int_{\mathbb{R}^{n}\times \mathbb{R}^{n}}[\varphi(x)+\psi(y)]\;d\pi(x,y)=\int_{\mathbb{R}^{n}}\varphi(x)d\mu(x)+\int_{\mathbb{R}^{n}}\psi(y)d\nu(y).
\]
\textit{\textbf{Kantorovich's problem}} is:  Minimize 
\begin{align}
\int_{\mathbb{R}^{n}\times \mathbb{R}^{n}} c(x,y)d\pi(x,y)=\int_{\mathbb{R}^{n}\times \mathbb{R}^{n}}|x-y|^{2}\;d\pi(x,y)
\end{align}
among all $\pi$ having marginals $\mu, \nu$. The set of such measures is denoted by $\Pi(\mu,\nu)$, and it is never empty, as it contains the product measure $\mu \otimes \nu$; also $\Pi(\mu,\nu)$ is convex. So this problem is actually a linear minimization problem with convex constraints. Whenever $\pi$ satisfies the marginal condition we say $\pi$ is admissible. A basic result in functional analysis applying continuity and compactness arguments is the existence of minimizers of functionals \cite{AmbGig2013}, \cite{santambrogio2015optimal}, and \cite{villani2003topics}.
\begin{remark}
As mentioned earlier, Kantorovich's problem is a relaxation of Monge's problem \cite{AmbGig2013}. To illustrate this, if $(Id, T): X \to X\times Y$ is defined by $(Id,T)(x):=(x,T(x))$  and $T_{\sharp}\mu=\nu$, then $(Id, T)_{\sharp}\mu \in \Pi(\mu,\nu)$, i.e., in a sense Kantorovich's problem contains Monge's.
\end{remark}

An important property we require of costs on paths is that of \textit{coercivity}, to prove existence of minimizers for (\ref{costpath}) for energy $c(\gamma)=\int_{0}^{1}L(\dot{\gamma}, \gamma, t)\;dt$ with Lagrangian $L(\dot{\gamma}, \gamma, t)=\frac{1}{2}|\dot{\gamma}|^{2}-V(\gamma(t), t)$; this is to, geometrically speaking, avoid paths that oscillate a lot. In this case the interpretation is that the cost $c(\gamma)$ would be rather large and thus would be too costly, and we wish to eschew this in our theory (see Condition \ref{condition:weaker} for further details).   

The optimal path (Kantorovich) problem is to minimize 
\begin{align}\label{costpath}
\inf_{\pi\in \Pi_{\text{path}}(\mu_{0},\mu_{1})}\int_{\Omega}c(\gamma)\;d\pi(\gamma),
\end{align}
this problem admits a solution under general conditions (Theorem \ref{result 1: Kantorovich on paths minimizers}). A similar result to Theorem \ref{result 1: Kantorovich on paths minimizers} applies to a new formulation in which we incorporate interacting paths (Theorem \ref{result 3: Kantorovich with interaction solutions}).
The former is analogues to a result that extends Brenier \cite{Brenier1991} and Gangbo and McCann \cite{gangbo1996geometry}. Indeed it turns out that these minimizers are given by maps (Theorem \ref{result 2: uniqueness Kantorovich solution on paths})
\[
\Gamma: X \times [0, 1] \longrightarrow \Omega,
\]
with the properties
\[
\Gamma(x,0)=x \quad \text{for all}\quad x\in X; \quad \Gamma(x,1)=T(x), \quad T_{\sharp}\mu_{0}=\mu_{1}.
\]
For a map $T$ that solves an auxiliary transport problem, the optimal measure will be given by $\pi_{\Gamma}:=(\Gamma)_{\sharp}\mu_{0}$ and solves the Kantorovich's problem (\ref{costpath}) uniquely.
Here,  $\mu_{0}$ is absolutely continuous with respect to Lebesgue; and, in turn $T$  solves the Monge's path dependent problem for the auxilary cost, $c_{\textbf{e}}(x,y)$, (\ref{effective infimum}).  For all $x$ in the support of $\mu_{0}$, the minimal path is of the form $t\mapsto\gamma_{x, y}(t)$, just like in the content of Theorem \ref{result 2: uniqueness Kantorovich solution on paths}. And thus for $T(x)$ the optimal map pushing $\mu_{0}$ forward to $\mu_{1}$ with respect to $c_{e}$, then the optimal map of paths is given by the composition $\Gamma(x,t):=\gamma_{x, T(x)}(t)$.

We also establish the dual problem in view of these settings to characterize the minimizers of the optimal path problem.

\subsection{Summary of main results}\label{Section: main results}
Let $\mathbb{R}^{n}$ be the $n$-dimensional Euclidean space and let $X\subset \mathbb{R}^{n}$ be a bounded domain.  The space of probability measures will be denoted by $\mathcal{P}(X)$. Consider the space of continuous paths,
\begin{align*}
\Omega:=\{\gamma: [0,1] \to X \;| \; \gamma\; \text{is continuous}\;\}.
\end{align*}
 If we endow $\Omega$ with a metric $\|\gamma_{1}-\gamma_{2}\|:=\max_{0\leq t \leq 1}\{d(\gamma_{1}(t),\gamma_{2}(t))\}$, then $\Omega$ becomes a complete metric space. In addition as $X$ is compact, $\Omega$ is a Polish. Let  $B_{R}(0):=\{y \in \mathbb{R}^{n}:\; \|x-y\|<R\}$ be an open ball. For all intents and purposes, we may simply take  $X:=\overline{B_{R}(0)}\subset \mathbb{R}^{n}$ for a large $R>0$.
Let $\mu_{0}$, $\mu_{1}\in \mathcal{P}(X)$. 

Consider probability measures $\pi \in \mathcal{P}(\Omega)$ with the following admissibility condition. 

\begin{definition}\label{evalmaps}
Define $e_{t}: \Omega \to X$, the \textit{evaluation} map by
\[
e_{t}(\gamma)=\gamma(t)\quad \text{for all }\quad 0\leq t \leq 1
\]

\end{definition}
In particular we have $e_{0}(\gamma)=\gamma(0)$ and $e_{1}(\gamma)=\gamma(1)$; this merely indicates the initial and final end-points of the path $\gamma$, respectively. The admissibility condition on $\pi$ is now given in the next definition.
\begin{definition}\label{admissiblemeasures}
Given $\mu_{0}, \mu_{1}\in \mathcal{P}(X)$, we say that $\pi\in \mathcal{P}(\Omega)$ is \textit{admissible} if the following holds
\begin{align*}
    (e_{0})_{\sharp}\pi=\mu_{0},\; (e_{1})_{\sharp}\pi=\mu_{1}
\end{align*}
\end{definition}
The set of all probability measures $\pi \in \mathcal{P}(\Omega)$ satisfying Definition \ref{admissiblemeasures} will be denoted by $\Pi_{\text{path}}(\mu_{0},\mu_{1})$.  Such measures $\pi$ are also known as dynamical couplings. This  notion is well known in the classical optimal transport literature, see Villani's discussion in \cite[Chapter 7]{villani2009optimal}. Moreover, the set $\Pi_{\text{path}}$ represents transport plans with associated paths and it is reminiscent of the standard admissible measures, in the classical Kantorovich measures $\Pi(\mu, \nu)$. We will revisit this description in Lemma \ref{c-cyclical monotone}, where we show a probability measure on the space of paths projects to a solution of the MKP in Euclidean space. 

Let $c: \Omega \to \mathbb{R}$ be a cost function. Define the linear functional 
$\mathcal{E}_{0}: \mathcal{P}(\Omega)\to \mathbb{R}$  by
\[
\mathcal{E}_{0}(\pi):=\int_{\Omega}c(\gamma)d\pi(\gamma).
\]
The \emph{optimal path problem} is then:

\textbf{Problem A.} Given $\mu_{0}, \mu_{1}\in \mathcal{P}(X)$, minimize $\mathcal{E}_{0}(\pi)$ among all $\pi \in \Pi_{\text{path}}(\mu_{0},\mu_{1})$.

Costs satisfying the following two conditions will prove essential:
\begin{Condition}\label{condition:weaker}
  The function $c:\Omega \to \mathbb{R}$ is bounded from below, lower semi-continuous, and it has the following coercivity property: given any two positive numbers $M$ and $N$, the set
  \begin{align*}
    \Omega_{M,N} := \{ \omega \in \Omega \mid |\omega(0)|\leq M,\; c(\omega) \leq N \},  
  \end{align*}
  is a compact subset of $\Omega$.
  
  \end{Condition}

\begin{Condition}\label{condition:stronger}
  The function $c:\Omega \to \mathbb{R}$ is of the form
  \begin{align*}
    c(\gamma) = \int_0^1 \tfrac{1}{2}|\dot \gamma|^2 - V(\gamma(t), t)\;dt,    
  \end{align*}
  where $V:\mathbb{R}^n\times [0, 1]\to\mathbb{R}$ is bounded and has spatial first and second derivatives bounded uniformly in $t$, with $\nabla_{x}V(x, t)$ satisfying an $L$-Lipschitz condition with $L<2/3$. 
\end{Condition}

 Our first Theorem (Theorem \ref{result 1: Kantorovich on paths minimizers}) is on existence under general conditions.

\begin{theorem}\label{result 1: Kantorovich on paths minimizers}
 Suppose $\mu_0$ and $\mu_1$ have compact support and that $c$ satisfies Condition \ref{condition:weaker}, then \textbf{Problem A} has at least one solution.  
\end{theorem}
 Now under stronger conditions, the optimal plan ends up being unique, and it is characterized by a general map solving Monge's problem, this follows both  Brenier's \cite{villani2009optimal} and Gangbo's and McCann's results \cite{gangbo1996geometry}:

\begin{theorem}\label{result 2: uniqueness Kantorovich solution on paths}
Suppose $\mu_0$ and $\mu_1$ have compact support, $\mu_0<<dx$, and $c$ satisfies Condition \ref{condition:stronger}, then there is at most one solution to \textbf{Problem A} given by an optimal map, $\Gamma: \text{spt}(\mu_{0}) \to \Omega$ (see Section \ref{Optimal plans given by maps}).
\end{theorem}

For the next problem we add an interaction term to $\mathcal{E}_{0}(\pi)$.  Given a continuous function $\mathcal{K}: \Omega\times \Omega \to \mathbb{R}$ define $\mathcal{U}: \Omega\times\mathcal{P}(\Omega)\longrightarrow \mathbb{R}$ by
\begin{align}\label{fancy u}
\mathcal{U}(\gamma, \pi):=\int_{\Omega}\mathcal{K}(\gamma,\sigma)\;d\pi(\sigma).
\end{align}
 Thus the new functional with interaction that we will study is the following:
\begin{align} \label{interaction term}
 \mathcal{E}(\pi):=\mathcal{E}_{0}(\pi)+\int_{\Omega}\mathcal{U}(\gamma,\pi)\;d\pi(\gamma)=\int_{\Omega}c(\gamma)d\pi(\gamma)+\int_{\Omega}\left(\int_{\Omega}\mathcal{K}(\gamma,\sigma)d\pi(\sigma)\right)d\pi(\gamma).
\end{align}
The function $\mathcal{K}(\gamma, \sigma)$ can be thought of as measuring interactions between $\gamma$ and $\sigma.$ Then $\mathcal{U}(\gamma, \pi)$ measures the total interaction between $\gamma$ and a distribution of paths given by $\pi$. The integral term $\int_{\Omega}\mathcal{U}(\gamma,\pi)\;d\pi(\gamma)$ in (\ref{MKpath}) is the ``total"  cost or energy from these interactions. 

The new optimal transportation problem with interacting paths is thus: 

\textbf{Problem B.}\label{Problem B} Given $\mu_{0}, \mu_{1}\in \mathcal{P}(X)$, minimize (\ref{MKpath}), given by (\ref{interaction term}), among all $\pi\in\Pi_{\text{path}}(\mu_{0}, \mu_{1})$.

This problem proved to be both interesting and subtle. Interesting because the objective functional (\ref{interaction term}) is not linear in $\pi$, and subtle because it might not even be convex in general. 
 
 As a particular example of interest, and for concreteness, consider $\mathcal{K}(\gamma, \sigma)$ given by an integral over the time interval $[0,1]$ of an exponential \footnote{The Coulomb kernel is also covered by our methods. See Appendix \ref{Appendix}} function: in general $\mathcal{K}(\gamma, \sigma)=\int_{0}^{1}\kappa(\gamma(t)-\sigma(t))\;dt$, where $\kappa:\mathbb{R}^n\to \,\mathbb{R}$ is of positive type,
 \begin{align}\label{exp path}
        \mathcal{K}(\gamma, \sigma):=\int_{0}^{1}\theta\exp\left\{-\beta|\gamma(t)-\sigma(t)|^{2}\right\}dt,
\end{align}
for $\beta>0$ and $\theta>0$. Note that if $\gamma$ and $\sigma$ are very close to each other, $\mathcal{K}(\gamma,\sigma)$ is very close to $0$. If, on the other hand,  $\sigma$ and $\gamma$ are a large distance  away from each other, then $\mathcal{K}(\gamma, \sigma)$ will be small.  More generally, Bochner's theorem implies $\mathcal{K}$ is convex if $\mathcal{K}(\gamma, \sigma)=\int_{0}^{1}\kappa(\gamma(t)-\sigma(t))\;dt$, where $\kappa$ is the Fourier transform of a finite, positive measure, as given in Reed's and Simon's \emph{Funtional Analysis I} book \cite[Theorem IX.9]{reed1972methods}. (See Appendix \ref{Appendix}).

\begin{remark}\label{Remark 1}
Let $d\pi_{t}(\cdot):=(e_{t})_{\sharp}\pi$. As $\kappa$ is the Fourier transform of a finite positive measure on $\mathbb{R}^n$, Bochner implies $H(\pi):=\int_{\mathbb{R}^n}\int_{\mathbb{R}^n}\kappa(x-y)d\pi_{t}(x)d\pi_{t}(y)$ is convex in $\pi$. As a result the quadratic functional term in (\ref{interaction term}) is convex in $\pi$. Indeed, the following computation shows it:
\begin{align*}
    \int_{\Omega}\int_{\Omega}\left(\int_{0}^{1}\kappa(\gamma(t)-\sigma(t))dt\right)d\pi(\sigma)d\pi(\gamma)&=\int_{0}^{1}\int_{\Omega}\int_{\Omega}\kappa(\gamma(t)-\sigma(t))d\pi(\sigma)d\pi(\gamma)dt\\
    (d\pi_{t}=(e_{t})_{\sharp}\pi)&=\int_{0}^{1}H(\pi)dt
\end{align*}
Consequently, the functional (\ref{interaction term}) is convex.
\end{remark}

\begin{remark}\label{Remark Coulomb}
For concreteness we will focus on the Gaussian interaction but other kernels are covered by our methods as well. For example,  $\kappa(x)=|x|^{2-n}$ ($n\geq 3$). The kernel $\mathcal{K}(\gamma, \sigma)=\int_{0}^{1}|\gamma(t)-\sigma(t)|^{2-n}\;dt$ is convex, as the Bochner-Schwartz theorem \cite[Theorem IX.10]{reed1972methods} implies $|\gamma(t)-\sigma(t)|^{2-n}$ is the Fourier transform of a positive measure of at most polynomial growth. The Coulomb force of interacting particles $\kappa(x-y)=|x-y|^{2-n}$ is an important class of examples that applies to our theory equally well for the interaction term. For Bochner's theorem, Bochner-Schwartz theorem, and the Coulomb kernel see Appendix \ref{Appendix}.
\end{remark} 

In the same spirit to Theorems \ref{result 1: Kantorovich on paths minimizers} and \ref{result 2: uniqueness Kantorovich solution on paths}, we study a new optimal transportation with interacting paths problem (\ref{MKpath}), and  show existence of minimizers of (\ref{interaction term}) and characterize the minimizers. For the remaining Theorems, we will only consider $\mathcal{K}(\gamma, \sigma)$ given by (\ref{exp path}). We note that, in a sense, \textbf{Problem B} contains \textbf{Problem A} as a special case. In our investigations we first analyzed \textbf{Problem A} and using this analysis as a foot-hold we  approached \textbf{Problem B.}

\begin{theorem}\label{result 3: Kantorovich with interaction solutions}
Suppose $\mu_{0}$ and $\mu_{1}$ have compact support and that $c$ satisfies Condition \ref{condition:weaker}, then the Kantorovich problem with interaction, \textbf{Problem B}, has at least one solution.
\end{theorem}

\begin{theorem}\label{thm: B implies A}
A solution to \textbf{Problem B} is a solution to \textbf{Problem A} with some effective cost.
\end{theorem}

\begin{theorem}\label{result 4: uniqueness Kantorovich solution with interaction}
Suppose $\mu_0$ and $\mu_1$ have compact support, $\mu_0<<dx$, $c$ satisfies Condition \ref{condition:stronger}, and $\mathcal{K}$ is as in (\ref{exp path}) for some $\theta$ and $\beta$. Then, there is $\theta_0>0$ depending on $\theta$ and $\beta$ such that if $\theta \in (0,\theta_0)$ then \textbf{Problem B} has a unique solution given by a map. 
\end{theorem}

\subsection{Outline of the paper}
Following the introduction and main results, Section \ref{section:minimal paths} looks at minimal paths, Lagrangians, and costs of paths; using Lagrangians with a potential. We supply several results indicating that an energy functional with a potential achieves its minimum on $\Omega$, and that for such a path minimizing an endpoint cost function $c_{e}$ with endpoints fixed is differentiable. Section \ref{Path MKP} deals with existence of minimizers and their properties. We prove Theorem \ref{result 1: Kantorovich on paths minimizers}, and produce the dual problem. Lastly, we prove Theorem \ref{result 2: uniqueness Kantorovich solution on paths}. The fourth section gives an account of the optimal path  with interaction effects. We prove existence of minimizers using a modulus of continuity argument and the coercivity property. We acquire the dual path dependence problem with the interaction effects. We also prove Theorems \ref{result 3: Kantorovich with interaction solutions} and \ref{result 4: uniqueness Kantorovich solution with interaction} in this section. Two appendices provide Bochner's statement on the convexity of the Coulomb potential and the other establishes the differentiability of the end point cost function (\ref{effective infimum}).

\section{Minimal paths}\label{section:minimal paths}  
In \cite{villani2009optimal}, Villani explains a construction in optimal transport that an \emph{action}, $\mathcal{A}$, which measures the cost of displacing along a continuous path $\gamma$, defined on a time interval, is used to consider a  cost, $c$, by minimizing the action among paths that go from the initial point of the path, $x$, to the final point of the path, $y$,
\begin{align*}
c(x, y)=\inf\{\mathcal{A}(\gamma):\gamma(0)=x,\;\gamma(1)=y,\;\gamma\in C([0, 1];\mathbb{R}^{n})\}.
\end{align*}
Similarly, we use this classical construction to define the endpoint cost function (\ref{effective infimum}). A typical and classical example of such an action is the kinetic energy, $\mathcal{A}(\gamma)=\int_{0}^{1}|\dot{\gamma}(t)|^{2}/2dt$. More generally, an action is given by Lagrangians as Villani details in \cite{villani2009optimal}. And following this line of thinking heavily, we begin our study of our cost action functional.
We cover some results pertaining a Lagrangian $L(\dot{\gamma},\gamma, t)$ and its action functional, which is given by the time integral of $L(\dot{\gamma}, \gamma, t)$ along the path, which defines a cost function $c(\gamma):=\int_{0}^{1}L(\dot{\gamma},\gamma, t)\; dt$, just like in Villani's general example of $\mathcal{A}(\gamma)$. The action functional will be interpreted as the cost of that path. We next consider a Lagrangian with a given potential $V(x, t)$.  All these preliminary propositions will be used in Section \ref{Path MKP} and Section \ref{Path MKP interaction}.

The following elementary proposition will be useful in what follows. It indicates how close a path is to a linear path.

\begin{proposition}\label{error prop}
Suppose $\gamma: [0, 1]\to \mathbb{R}^{n}$ is $C^{2}$ and $\delta>0$ is such that 
\begin{align*}
    |\ddot{\gamma}(t)|\leq \delta\quad\forall t\in [0, 1].
\end{align*}
Then for every $t \in [0, 1]$ we have 
\begin{align*}
    |\gamma(t)-(\gamma(0)+t[\gamma(1)-\gamma(0)])|\leq \sqrt{n}\delta\quad \text{and}\quad |\dot{\gamma}(t)-(\gamma(1)-\gamma(0))|\leq \sqrt{n}\delta.
\end{align*}
\end{proposition}

\begin{proof}
 Let $\mathbf{e}(t):=\gamma(t)-(\gamma(0)+t(\gamma(1)-\gamma(0)))$. Then $\mathbf{e}(0)=\mathbf{e}(1)=\mathbf{0}$, and $\mathbf{e}(t)$ is continuously differentiable on $[0, 1]$. Furthermore, $\ddot{\mathbf{e}}(t)=\ddot{\gamma}(t)$ and $|\ddot{\mathbf{e}}(t)|\leq \delta$.
Therefore for any $0\leq t_{1}<t_{2}\leq 1$, we have 
 \begin{align*}
     \dot{\mathbf{e}}(t_{2})-\dot{\mathbf{e}}(t_{1})=\int_{t_{1}}^{t_{2}}\ddot{\mathbf{e}}(t)dt.
 \end{align*}
 So $|\dot{\mathbf{e}}(t_{2})-\dot{\mathbf{e}}(t_{1})|\leq \delta$ for all $0\leq t_{1}<t_{2}\leq 1$.
 
 In coordinate components of $\mathbf{e}(t)$, we will show that for each $i=1, \ldots, n$ there is some $t_{i}\in [0, 1]$ such that $\dot{\mathbf{e}_{i}}(t_{i})=0$. Thus, in this case, $|\dot{\mathbf{e}_{i}}(t)|\leq \delta$ for all $t\in [0, 1].$ Then 
 \begin{align*}
     |\dot{\mathbf{e}}(t)|\leq\sqrt{\delta^{2}+\cdots+\delta^{2}}= \sqrt{n}\delta.
 \end{align*}
 According to the mean value theorem, applied to $\mathbf{e}_{i}(t_{i})$, there exists  $\overline{t}_{i}\in [0, 1]$ such that 
 \begin{align*}
 0=\mathbf{e}_{i}(1)-\mathbf{e}_{i}(0)=\dot{\mathbf{e}}_{i}(\overline{t}_{i}).
 \end{align*}
 Therefore for all $t\in [0, 1]$,
 $\mathbf{e}(t)=\mathbf{e}(0)+\int_{0}^{t}\dot{\mathbf{e}}(s)ds=\int_{0}^{t}\dot{\mathbf{e}}(s)ds$. Then $|\mathbf{e}(t)|\leq\sqrt{n}\delta$ and the proposition follows at once.
 
\end{proof}

We will focus on Lagrangians of the form 
\begin{align*} 
L(\dot{\gamma}(t), \gamma(t), t):=\frac{1}{2}|\dot{\gamma}(t)|^{2}-V(\gamma(t), t),
\end{align*}
where $V:\mathbb{R}^{n}\times [0, 1]\to\mathbb{R}$ is a continuously differentiable potential which is bounded from below. Then we consider the cost function 
\begin{align*}
c(\gamma)=\int_{0}^{1}L(\dot{\gamma}(t),\gamma(t), t)\;dt.
\end{align*}
We see that this corresponds to the classical Lagrangian $L(x, v)=\frac{1}{2}|v|^{2}-V(x, t)$, where $v$ is the velocity (or time-derivative) of the path $\gamma$ at $t$, with endpoints $x$ and $y$ of $\gamma$ fixed and $V(x, t)$ some potential, just like in Villani's Lagrangian example in \cite{villani2009optimal}. It is well known that if $V(x, t)\in C^{1}$, the minimizers of $L(x,v)$ with endpoints fixed satisfy Newton's dynamical equation 
\begin{align*}
    \frac{d^{2}x}{dt^{2}}=-\nabla_{x} V(x, t).
\end{align*}

Everything that follows can be done for more general Lagrangians,  but we will focus on the Lagrangian above  for the sake of concreteness. 

\begin{proposition}\label{minimal paths plus potential}
Let $\gamma: [0, 1]\to \mathbb{R}^{n}$ be a Lipschitz path which is twice differentiable in $(0, 1)$. Then the minimizers of $c(\gamma)$, with endpoints $x$ and $y$ of $\gamma$ fixed, satisfy the equation 
\begin{align*}
    \ddot{\gamma}(t)=-\nabla_{x}V(\gamma(t), t)
\end{align*}
\end{proposition}

\begin{proof}
This is just the Euler-Lagrange equation applied to the cost functional, see Evans'
\cite[Ch 8]{evans1998partial}.
\end{proof}

Notice that when $V\equiv 0$,  we get $\ddot{\gamma}=0$. If $V$ is small enough, then we can expect that minimal $\gamma$'s are close to straight lines. We now quantify this intuition.

\begin{corollary}\label{corollary to prop1 and prop 2}
Suppose the Lagrangian is given by $L(\dot{\gamma}, \gamma, t)=\frac{1}{2}|\dot{\gamma}(t)|^{2}-V(\gamma(t), t)$ and $V: \mathbb{R}^{n}\times [0, 1]\to \mathbb{R}$ is continuously differentiable and bounded from below. If $\delta>0$ is such that $\|\nabla V\|_{\infty}\leq\delta, $ then if $\gamma$ is minimal,
\begin{align*}
    |\gamma(t)-(\gamma(0)+t[\gamma(1)-\gamma(0)])|\leq \sqrt{n}\delta\quad \text{and}\quad |\dot{\gamma}(t)-(\gamma(1)-\gamma(0))|\leq \sqrt{n}\delta.
\end{align*}
\end{corollary}

\begin{proof}
Choose a minimizer $\gamma$ of $c(\gamma)$. Since $\delta>0$ is such that $|\nabla V(\gamma(t), t)|\leq \delta$ for all $t\in [0, 1]$, Proposition \ref{minimal paths plus potential} applies to show that $|\ddot{\gamma}(t)|\leq \delta$ for all $t\in [0, 1]$. Then applying Proposition \ref{error prop} establishes the corollary.
\end{proof}

Let us show now that for any pair of points $x, y$, there is a unique minimal $\gamma$ between them, provided $V$ satisfies a smallness condition.

\begin{proposition}\label{boundary value problem}
 Let $V$ be bounded and of class $C^{2}$ and such that $\nabla V(x, t)$ satisfies the spatial $L$-Lipschitz condition for $L<1$. Then for all $x, y\in \mathbb{R}^{n}$ there is a unique path $\gamma_{x, y}(\cdot)$ which minimizes $c(\gamma)$ among all paths from $x$ to $y$, and this function is Lipschitz continuous in $x, y$.
\end{proposition}

\begin{proof}
Let $\omega(t)$ be a solution for the boundary value problem
 \begin{align*}
     \ddot{\gamma}(t)&=-\nabla_{x} V(\gamma(t), t),\\
     \gamma(0)&=x, \\
     \gamma(1)&=y.
     \end{align*} 
Then $\omega(t)$ will satisfy the integral equation
\begin{align*}
    \omega_{x,y}(t)=(1-t)x+ty+t\int_{0}^{1}(1-s)\nabla V(\omega(s),s)\;ds-\int_{0}^{t}(t-s)\nabla V(\omega(s),s)\;ds.
\end{align*}

To establish the uniqueness condition and Lipschitz continuity, fix two pair of points, $x_{1}, y_{1}$ and $x_{2}, y_{2}$.  Suppose two such solutions exist, call them $\gamma_{x_{1},y_{1}}(t)$ and $\gamma_{x_{2},y_{2}}(t)$. Then both satisfy the integral equation. Denote $d(\gamma_{x_{1},y_{1}}(t),\gamma_{x_{2},y_{2}}(t)):=\sup_{[0, \tau]}|\gamma_{x_{1},y_{1}}(t)-\gamma_{x_{2}, y_{2}}(t)|$.
 Then we have
 
 \begin{align*}
     |\gamma_{x_{1}, y_{1}}(t)-\gamma_{x_{2},y_{2}}(t)|\leq |(1-t)(x_{1}-x_{2})&+t(y_{1}-y_{2})|\\
     &+t\int_{0}^{1}(1-s)|\nabla V(\gamma_{x_{1},y_{1}}(s), s)-\nabla V(\gamma_{x_{2}, y_{2}}(s), s)|\;ds\\
     &+\int_{0}^{t}(t-s)|\nabla V(\gamma_{x_{1},y_{1}}(s), s)-\nabla V(\gamma_{x_{2}, y_{2}}(s), s)|\;ds
 \end{align*}
 \begin{align*}
  \quad \quad \quad\leq |x_{1}-x_{2}|+|y_{1}-y_{2}|&+L\int_{0}^{1}(1-s)|\gamma_{x_{1},y_{1}}(s)-\gamma_{x_{2}, y_{2}}(s)|\;ds\\
     &+L\int_{0}^{t}(t-s)|\gamma_{x_{1},y_{1}}(s)-\gamma_{x_{2},y_{2}}(s)|\;ds
 \end{align*}
 \begin{align*}
     \quad\quad\quad\quad\quad\quad\quad\leq |x_{1}-x_{2}|+|y_{1}-y_{2}|&+L\int_{0}^{1}(1-s)\;ds\sup_{[0, t]}|\gamma_{x_{1},y_{1}}(s)-\gamma_{x_{2}, y_{2}}(s)|\\
     &+L\sup_{[0, t]}\int_{0}^{t}(t-s)\;ds\;\sup_{[0, t]}|\gamma_{x_{1}, y_{1}}(s)-\gamma_{x_{2}, y_{2}}(s)|.
 \end{align*}
 Noticing that the term with the common factor equals
\begin{align*}
 \int_{0}^{1}(1-s)\;ds+\sup_{0\leq t \leq 1}\int_{0}^{t}(t-s)\;ds =1,
\end{align*}
and after rearranging, the above estimate equals the estimate, since $L<1$ by assumption,
\begin{align*}
    d(\gamma_{x_{1},y_{1}}, \gamma_{x_{2}, y_{2}})&\leq |x_{1}-x_{2}|+|y_{1}-y_{2}|+L\;d(\gamma_{x_{1}, y_{1}},\gamma_{x_{2}, y_{2}})\\
    \Longleftrightarrow &\; d(\gamma_{x_{1},y_{1}}, \gamma_{x_{2}, y_{2}})\leq \frac{1}{1-L}(|x_{1}-x_{2}|+|y_{1}-y_{2}|),
\end{align*}
we are done with the proof of the proposition.
\end{proof}

For a given cost $c: \Omega \to \mathbb{R}$, we introduce the \emph{endpoint} function between an initial point $x$ and a final point $y$, which is obtained by minimizing the linear functional among paths $\gamma$ that go from $x$ to $y$. So for any pair of points $x,y\in \mathbb{R}^{n}$, let

\begin{align}\label{effective infimum}
    c_{\textbf{e}}(x, y):=\inf_{\gamma(0)=x,\; \gamma(1)=y}c(\gamma).
\end{align}

\begin{lemma}\label{effective cost is differentiable}
Let $c_{\textbf{e}}(x, y)$ be as in (\ref{effective infimum}). Then $c_{\textbf{e}}$ is differentiable with respect to $x$ and $y$. Moreover, $c_{\textbf{e}}(x, y)=c_{\textbf{e}}(y, x)$, and $\nabla_{y}c_{\textbf{e}}(x, y)=y-x-\int_{0}^{1}t \nabla_{x} V(\gamma_{x, y}(t), t)\;dt.$
\end{lemma}

\begin{proof}
Let us first show the symmetric condition. The endpoint  cost $c_{\textbf{e}}(y, x)$ between $y$ and $x$ is acquired by minimizing the linear functional among paths that go from $y$ to $x$. Let us denote such path as $\overline{\gamma}(t)$. More concretely, let $s: [0, 1] \to [0, 1]$ be given by $s(t)=1-t$ for all $t\in [0, 1].$ Then $\overline{\gamma}(t)=\gamma(s(t))$ is a path from $y$ to $x$. Let $1-t_{0}=s(r_{0})$ where $r_{0}\in [0, 1]$ is in the domain of $s$. Then
\begin{align*}
    \dot{\overline{\gamma}}(r)=\frac{d}{dr}\overline{\gamma}&=\frac{d}{dr}\gamma(s(r))\\
    &=\frac{d \gamma}{dt}\frac{ds}{dr}\\
    &=-\dot{\gamma}(t).
\end{align*}
So that then using the change of variables $s(t)=1-t$, one can see $c(\gamma)=c(\overline{\gamma})$, and then
\begin{align*}
    c_{\textbf{e}}(x, y)=c_{\textbf{e}}(y,x).
\end{align*}

To show the differentiability of $c_{\textbf{e}}(x, y)$ at $(x, y)$ as a function of $x$ and $y$, by the symmetry condition, it suffices to only show it is differentiable with respect to $y$. Indeed, pick a minimal path $\gamma_{0}$ from $\gamma_{0}(0)=x$ to $\gamma_{0}(1)=y$. In this case we have $c_{\textbf{e}}(x, y)=c(\gamma_{0})$. Consider another path, not necessarily minimal, $\gamma_{h}(t):=\gamma_{0}(t)+th\cdot \hat{e}$ from $x$ to $y+h\hat{e}$, where $h\neq 0$ and $\hat{e}:=(0,\ldots, 1, \ldots,0)$. From this, we know $c_{\textbf{e}}(x, y+h\hat{e})\leq c(\gamma_{h}(t)).$ Putting this together, we get 
\begin{align}\label{ineq: quotient difference limit}
    \lim_{h\to 0}\frac{c_{\textbf{e}}(x, y+h\hat{e})-c_{\textbf{e}}(x, y)}{h}\leq \lim_{h \to 0}\frac{c(\gamma_{h}(t))-c(\gamma_{0}(t))}{h}.
\end{align}
Now, since the cost of the path $\gamma_{h}(t)$ can be  differentiated ($c(\gamma_{h})$ is explicit) we know that the right-hand side of (\ref{ineq: quotient difference limit}) is well-defined. That the left-hand side limit exists\footnote{We prove this fact in Appendix \ref{Appendix2}} follows from a standard argument (using compactness and uniqueness of solutions to linear ODE's). Then
it follows $\langle\nabla_{y}c_{\textbf{e}}(x, y), \hat{e}\rangle\leq \frac{d}{dh}|_{h=0}c(\gamma_{h}(t))$. To establish the reverse inequality, consider another path $\gamma_{h}(t)=\gamma_{0}(t)+th(-\hat{e})$ from $x$ to $y-h\hat{e}$. From this, (\ref{ineq: quotient difference limit}) becomes

\begin{align}\label{ineq: quotient difference less than}
    \lim_{h \to 0}\frac{c_{\textbf{e}}(x, y-h\hat{e})-c_{\textbf{e}}(x, y)}{h}\leq \lim_{h \to 0}\frac{c(\gamma_{h}(t))-c(\gamma_{0}(t))}{h}.
\end{align}
A change of variables with $z:=y-h\hat{e}$; noticing $z \to y$ as $h \to 0$, the left hand side of (\ref{ineq: quotient difference less than}) yields the reverse inequality, so that
\begin{align*}
    -\lim_{h \to 0}\frac{c_{\textbf{e}}(x, y+h\hat{e})-c_{\textbf{e}}(x, y)}{h}=\langle -\nabla_{y}c_{\textbf{e}}(x,y), \hat{e}\rangle,
\end{align*}
while the right hand side of (\ref{ineq: quotient difference less than}) with $\gamma_{h}(t)=\gamma_{0}(t)-th\cdot\hat{e}$ gives
\begin{align*}
    \frac{d}{dh}\Big|_{h=0}c(\gamma_{h}(t))=\left\langle-\int_{0}^{1}\Big(\dot{\gamma_{0}}(t)-\nabla_{x} V(\gamma_{0}(t), t)t\Big)\;dt, \hat{e}\right\rangle.
\end{align*}
Putting this together in (\ref{ineq: quotient difference limit}) gives the reverse inequality and thus  
\begin{align*}
    \left\langle \nabla_{y}c_{\textbf{e}}(x, y), \hat{e}\right\rangle=
\frac{d}{dh}|_{h=0}c(\gamma_{h})=\left\langle\int_{0}^{1}\Big(\dot{\gamma}_{0}(t)-\nabla_{x} V(\gamma_{0}(t), t)t\;dt\Big),\hat{e}\right\rangle.
\end{align*}
Therefore the differentiability of $c_{\textbf{e}}(x, y)$ with respect to $y$ follows and therefore $\nabla_{y}c_{\textbf{e}}(x,y)=y-x-\int_{0}^{1}t\;\nabla_{x} V(\gamma_{0}(t), t)\;dt$.
\end{proof}

Finally, we show $\nabla c_{\textbf{e}}(x, y)$ satisfies the twist condition from \cite{villani2009optimal}. 

\begin{lemma}\label{twist condition}
Let $V$ be bounded and of class $C^{2}$ and let $\nabla V$ be $L$-Lipschitz continuous with $L<2/3$. For all $y \in \mathbb{R}^{n}$, and $x_{1}\neq x_{2}$, we have 
\begin{align*}
    \nabla_{y}c_{\textbf{e}}(x_{1},y)\neq\nabla_{y}c_{\textbf{e}}(x_{2}, y)
\end{align*}
\end{lemma}

\begin{proof}
Lemma \ref{effective cost is differentiable} tells us 
\begin{align*}
    \nabla_{y}c_{\textbf{e}}(x, y)=y-x-\int_{0}^{1}t\;\nabla_{x} V(\gamma_{x, y}(t), t)\;dt.
\end{align*}
Now for all $y$, 
\begin{align*}
    \nabla_{y}c_{\textbf{e}}(x_{1}, y)-\nabla_{y}c_{\textbf{e}}(x_{2}, y)=x_{2}-x_{1}-\int_{0}^{1}t\Big(\nabla_{x_{1}} V(\gamma_{x_{1}, y}(t), t)-\nabla_{x_{2}} V(\gamma_{x_{2}, y}(t), t)\Big)\;dt.
\end{align*}
The reverse triangle inequality and invoking Proposition \ref{boundary value problem}, most notably the bound at the end of its proof and the Lipschitz condition on $\nabla V$, we have

\begin{align*}
    |\nabla_{y}c_{\textbf{e}}(x_{1}, y)-\nabla_{y}c_{\textbf{e}}(x_{2}, y)|&\geq |x_{1}-x_{2}|-L\int_{0}^{1}t|\gamma_{x_{1}, y}(t)-\gamma_{x_{2}, y}(t)|\;dt\\
    &\geq |x_{1}-x_{2}|-L\Big(\frac{1}{1-L}\Big)|x_{1}-x_{2}|\int_{0}^{1}t\;dt\\
    &=|x_{1}-x_{2}|\Big(1-\frac{1}{2}\frac{L}{1-L}\Big)\\
    &>0.
\end{align*}
Since $L<2/3$.
\end{proof}

\section{Minimizers and their properties} \label{Path MKP}

\subsection{Existence of minimizers} \label{Existence of minimizers}
Before showing that minimizers of a linear functional on paths \textit{do} exist, a few helpful results are in order.  An important tool in measure theory is compactness of measures. For instance in a compact metric space $X$, if  $\{\mu_{k}\}$ is a sequence of probability measures bounded above by a constant $C>0$,
\[
\mu_{k}(X)\leq C \quad \text{for all }\quad k,
\]
then one can find a subsequence  of  measures $\{\mu_{k_j}\}$ and  another nonnegative  measure $\mu$ such that $\mu_{k_j}\rightharpoonup \mu$ as $j \to \infty$, i.e. $\int \phi\;d\mu_{k_j}\to \int \phi \;d\mu$ as $j \to \infty$ for all $\phi\in C(X).$ A more general result due to Prokhorov, Theorem \ref{Prokhorov}, includes this line of reasoning. First, however, we need the following definition.

\begin{definition}\label{tight}
Let $X$ be a metric space and let $\{\mu_{k}\}$ be a sequence of nonnegative Borel measures in $X$. The sequence is said to be \textit{tight} if for each $\varepsilon>0$, there exists a compact set $X_{\varepsilon}\subset X$ such that 
\[
\mu_{k}(X\setminus X_{\varepsilon})\leq \varepsilon \quad \text{for all }\quad k.
\]
\end{definition}

Prokhorov's theorem, which can be found in either \cite{villani2003topics} or \cite{santambrogio2015optimal}, is the following.

\begin{theorem}[Prokhorov's Theorem]\label{Prokhorov} 
Let $X$ be a separable metric space, and let $\{\mu_{k}\}$ be a sequence of nonnegative Borel measures. If the sequence $\{\mu_{k}\}$ is tight, then there exists a subsequence $\{\mu_{k_j}\}$ and another nonnegative Borel measure $\mu$ such that 
\[
\mu_{k_j}\rightharpoonup \mu \quad \text{as}\quad j \to \infty.
\]
\end{theorem}

We need two additional propositions before we prove existence of minimizers. And Condition \ref{condition:weaker}  will help us with the proof  of Theorem \ref{result 1: Kantorovich on paths minimizers} as we will momentarily see.

\begin{proposition}\label{prop condition 2 implies condition 1}
Condition \ref{condition:stronger} implies Condition \ref{condition:weaker}. 
\end{proposition}

\begin{proof}
 
 We have $c(\gamma):=\int_{0}^{1}\frac{1}{2}|\dot{\gamma}(t)|^{2}-V(\gamma(t), t)\;dt$ is bounded from below and lower semi-continuous. To be more concrete, Proposition \ref{lower semicontinuity of c} establishes the lower semi-continuity part. That $c(\gamma)$ is bounded from below follows from
 \begin{align*}
     \int_{0}^{1}\frac{1}{2}|\dot{\gamma}(t)|^{2}-V(\gamma(t), t)\;dt\geq-K,
 \end{align*}
 as $V$ is bounded by $K>0$.
 Recall the set  $\Omega_{M, N}$ from Condition \ref{condition:weaker}. We will show this set is compact for all $N>0 $. To this end, we will show $\Omega_{M, N}$ is uniformly bounded and uniformly continuous, and then invoke Arzela-Ascoli.  We first show the latter. Given $\varepsilon>0$, fix $0\leq t_{2}<t_{1}\leq 1$.  We have to show $|\gamma(t_{1})-\gamma(t_{2})|\leq \varepsilon$. Indeed, up to a factor of $\frac{1}{2}$, we have
 \begin{align*} 
|\gamma(t_{1})-\gamma(t_{2})|\leq\int_{t_{2}}^{t_{1}}|\dot{\gamma}(t)|\;dt&\leq \left(\int_{0}^{1}|\dot{\gamma}(t)|^{2}\;dt\right)^{1/2}\left(\int_{t_{2}}^{t_{1}}1\;dt\right)^{1/2}\\
&=\left(\int_{0}^{1}|\dot{\gamma}(t)|^{2}-V(\gamma(t), t)+V(\gamma(t), t)dt\right)^{1/2}\;\left|t_{1}-t_{2}\right|^{1/2}\\
&\leq \left(\int_{0}^{1}|\dot{\gamma}(t)|^{2}-V(\gamma(t), t)dt+K\right)^{1/2}\;\left|t_{1}-t_{2}\right|^{1/2}\\
&\leq \sqrt{N+K} \;|t_{1}-t_{2}|^{1/2}.
\end{align*} 
This shows $\gamma$ has a bounded Hölder seminorm. Now we check $\Omega_{M, N}$ is uniformly bounded.  This follows from the fact that $|\gamma(0)|\leq M$ for every $\gamma\in \Omega_{M,N}$ and the modulus of continuity given above.  
\end{proof}

\begin{proposition}\label{lower semicontinuity of c}
The cost function $c(\gamma)$ given by Condition \ref{condition:stronger} is lower semi-continuous.
\end{proposition}

\begin{proof}
Assume $\gamma_{n}\to \gamma$ in the sup-norm in $[0, 1]$, we are going to show $\liminf_{n\to \infty}c(\gamma_{n})\geq c(\gamma).$
The cost function $c(\gamma)$ has two terms. The first term satisfies 
\begin{align*}
\liminf_{n\to \infty}\int_{0}^{1}|\dot{
\gamma_{n}}(t)|^{2}\;dt\geq \int_{0}^{1}|\dot{\gamma}(t)|^{2}\;dt
\end{align*}
from \cite[Ch8.2.2, Theorem 1]{evans1998partial}, as $\gamma$ is Lipschitz continuous and $|\dot{\gamma}|^{2}$ smooth, convex and bounded below. While the second term has the potential $V$ which is uniformly continuous; this means that given $\varepsilon>0$, there is $\delta>0$ such that
\begin{align*}
    \|\gamma_{n}-\gamma\|\leq \delta \implies |V(\gamma_{n}, t)-V(\gamma, t)|\leq \varepsilon
\end{align*}
in $[0, 1]$. In turn this then means that the second term satisfies $\lim_{n \to \infty}\int_{0}^{1}V(\gamma_{n}, t)\;dt=\int_{0}^{1}V(\gamma, t)\;dt$. Putting this together with the above liminf establishes the result.
\end{proof}

We are now ready to prove Theorem \ref{result 1: Kantorovich on paths minimizers}.

  \begin{proof}[\textbf{Proof of Theorem \ref{result 1: Kantorovich on paths minimizers}}]
 Since $c(\gamma)\geq -K$, then $\mathcal{E}_{0}(\pi)\geq -K$ for all $\pi \in \mathcal{P}(\Omega)$, thus the infimum of $\mathcal{E}_{0}$ in $\Pi_{\text{path}}(\mu_{0},\mu_{1})$ is finite. Let $\{\pi_{k}\}_{k}$ be a minimizing sequence $\mathcal{E}_{0}(\pi_{k})\longrightarrow\inf_{\pi \in \Pi_{\text{path}}(\mu_{0},\mu_{1})}\mathcal{E}_{0}(\pi).$
 Note that since $\{\pi_{k}\}_{k}$ is minimizing, we can find a constant $C>0$ for which $\mathcal{E}_{0}(\pi_{k})\leq C$ for all $k$. 
 
 Suppose $M>0$ is such that $\text{spt}\;(\mu_{0})\subset B_{M}(0)$. Define for $N>0,$  $X_{N}:=\{\gamma\in \Omega\;|\;c(\gamma)\leq N, N>0\}$ then, for every admissible  $\pi$ we have thus
 \begin{align*}
     \pi(X_{N})=\pi(\Omega_{M,N}).
 \end{align*}
 One can see this from the fact that $(e_{0})_{\sharp}\pi=\mu_{0}$ implies that $\text{spt}\;(\pi)\subset e_{0}^{-1}(\text{spt}\;(\mu_{0}))$. So the equality follows from $\pi$ being supported in the set $\{\gamma\in \Omega:\;|\gamma(0)|\leq M\}.$ In particular, 
 \begin{align*}
     \pi(\Omega_{M, N}^{c})=1-\pi(\Omega_{M, N})=1-\pi(X_{N})=\pi(X_{N}^{c}).
 \end{align*}
 We have $c(\gamma)>N$ in $X_{N}^{c}$, so 
 \begin{align*}
     \pi(X_{N}^{c})=\frac{1}{N}\int_{X_{N}^{c}}N\;d\pi(\gamma)\leq \frac{1}{N}\int_{X_{N}^{c}}c(\gamma)\;d\pi(\gamma),
 \end{align*}
 which implies $\mathcal{E}_{0}(\pi)/N\geq \pi(X_{N}^{c})$. Applying this to each $\pi_{k}$,
 \begin{align}\label{coercive}
     \pi_{k}(\Omega_{M, N}^{c})\leq \frac{1}{N}\mathcal{E}_{0}(\pi_{k})\leq\frac{C}{N} \quad \text{for all }\quad k
 \end{align}
 
Returning to inequality (\ref{coercive}), for each $N>0 $ given $\varepsilon:=\frac{C}{N}>0$,
\[
\pi_{k}\Big(\Omega_{M, N}^{c}\Big)\leq \frac{C}{N}=\varepsilon \quad \text{for all} \quad k.
\]
 By Condition \ref{condition:weaker}, $\Omega_{M, N}$ is compact, then $\{\pi_{k}\}_{k}$ is tight. Hence Prokhorov (Theorem \ref{Prokhorov}) says there exists a subsequence $\{\pi_{k_j}\}_{j}\subset \Pi_{\text{path}}(\mu_{0},\mu_{1})$ and another Borel probability measure $\rho$ such that 
\[
\pi_{k_j}\rightharpoonup \rho \quad \text{as}\quad j \to \infty.
\]
The claim is that then $\rho \in \Pi_{\text{path}}(\mu_{0},\mu_{1})$. Recall the change of variables formula (\ref{changevariables}). We will use an extension of this with the evaluation map to show $\rho$ lies in $\Pi_{\text{path}}(\mu_{0},\mu_{1}).$ Given any test function $\phi \in C_{b}(\Omega)$, 
\begin{align*}
\int_{\Omega}\phi(x)\;d((e_{0})_{\sharp}\rho(x))=\int_{\Omega}\phi(\gamma(0))\;d\rho(\gamma)&=\lim_{j \to \infty}\int_{\Omega}\phi(\gamma(0))d\pi_{k_j}(\gamma)\\
&=\lim_{j\to\infty}\int_{\Omega}\phi(\gamma)d((e_{0})_{\sharp}\pi_{k_j})\\
&=\int_{\Omega}\phi(\gamma)\;d\mu_{0}.
\end{align*}
From the arbitrariness of $\phi \in C_{b}(\Omega)$ we get $(e_{0})_{\sharp}\rho=\mu_{0}$. Similarly, by the same argument we can show $(e_{1})_{\sharp}\rho=\mu_{1}$. Thus, $\rho \in \Pi_{\text{path}}(\mu_{0},\mu_{1})$.

To conclude, we must show $\mathcal{E}_{0}(\rho)\leq\liminf_{j\to \infty}\mathcal{E}_{0}(\pi_{k_j})$. By Proposition \ref{lower semicontinuity of c}, $c(\gamma)$ is nonnegative and lower semi-continuous. Then it is well known \cite[Ch 4]{villani2009optimal} that $c$ can be written as the limit of a nondecreasing, sequence of bounded continuous functions $\{c_{n}\}_{n\geq 0}$ for $c_{n}\leq c$.  The monotone convergence theorem shows
\begin{align*} 
\mathcal{E}_{0}(\rho)=\int_{\Omega}c(\gamma)\;d\rho(\gamma)&= \lim_{n \to \infty}\int_{\Omega}c_{n}(\gamma)\;d\rho(\gamma)\\
&=\lim_{n\to \infty}\lim_{j \to \infty}\int_{\Omega}c_{n}(\gamma)\;d\pi_{k_j}(\gamma)\\
&\leq \liminf_{j\to\infty}\int_{\Omega}c(\gamma)\;d\pi_{k_j}(\gamma).
\end{align*}
We are done since $\mathcal{E}_{0}(\rho)\leq \liminf_{j\to\infty}\int_{\Omega}c(\gamma)\;d\pi_{k_j}(\gamma)=\inf_{\pi\in \Pi_{\text{path}}(\mu_{0},\mu_{1})}\mathcal{E}_{0}(\pi)$
but the definition of infimum  $\inf_{\pi\in \Pi_{\text{path}}(\mu_{0},\mu_{1})}\mathcal{E}_{0}(\pi)\leq \mathcal{E}_{0}(\rho)$ yields,  

\[
\mathcal{E}_{0}(\rho)=\inf_{\pi\in \Pi_{\text{path}}(\mu_{0},\mu_{1})}\mathcal{E}_{0}(\pi),
\]
and so $\mathcal{E}_{0}(\pi)$ attains its infimum at $\rho.$
 \end{proof}
  
  \subsection{The dual problem and the endpoint cost function}\label{The dual problem and the endpoint cost function}
  In this section we are going to study the \textit{endpoint cost function} introduced in Section \ref{section:minimal paths}, equation (\ref{effective infimum}). Namely,
 \begin{align*}
 c_{\textbf{e}}(x,y) & :=\inf\{c(\gamma)\;| \;\gamma(0)=x \quad \text{and}\quad \gamma(1)=y\}. \label{eqn:cost aux definition}
 \end{align*}
\textbf{Examples.}\\ \textbf{1.} If $c(\gamma):=\int_{0}^{1}\frac{1}{2}|\dot{\gamma}(t)|^{2}\;dt$, then $c_{\textbf{e}}(x,y)=\frac{1}{2}|x-y|^{2}$.\\
\textbf{2.} If $c(\gamma):=\int_{0}^{1}\frac{1}{p}|\dot{\gamma}(t)|^{p}\;dt, \; p\geq 1$, in $\mathbb{R}^{n}$, then $c_{\textbf{e}}(x,y)=\frac{1}{p}|x-y|^{p}$.\\

It is not too difficult to show this. In fact we can find in Villani's books \cite{villani2003topics}, \cite{villani2009optimal}:\\ 
\emph{Claim. If $h$ is a convex function defined on $\mathbb{R}^{n}$, then $\inf\{\int_{0}^{1}h(\dot{\gamma}(t))\;dt\;:\; \gamma(0)=x, \gamma(1)=y\}=h(y-x)$}.

We have $\int_{0}^{1}h(\dot{\gamma}(t))\;dt\geq h\Big(\int_{0}^{1}\dot{\gamma}(t)dt\Big)=h(y-x)$, by Jensen's inequality \cite{villani2003topics}, \cite{villani2009optimal}.

 The Kantorovich problem associated to paths is the problem  of minimizing the linear functional $\pi \mapsto \int c(\gamma )\;d\pi$ subject to the linear constraints $(e_{0})_{\sharp}\pi=\mu_{0}$, $(e_{1})_{\sharp}\pi=\mu_{1}$, and $\pi\geq 0$. Linear minimization problems with convex constraints of this type admit a natural dual problem, see Solomon's book \cite[Ch 10]{solomon2015numerical}. We now define this dual problem, as is usually done, using Lagrange multipliers. 
 
For any pair $(\phi, \psi)\in L^{1}(d\mu_{0})\times L^{1}(d\mu_{1})$ and $\pi \in \mathcal{P}(\Omega)$, we define
  \begin{align*}
       L(\pi; \phi, \psi):= \int_{\Omega}c(\gamma)\;d\pi(\gamma)&+\int_{X}\phi(x)d\mu_{0}(x)-\int_{\Omega}\phi(\gamma(0))d\pi(\gamma)\\
          &+\int_{X}\psi(x)d\mu_{1}(y)-\int_{\Omega}\psi(\gamma(1))d\pi(\gamma)
      \end{align*}
  
  Then after some rearrangement 
   \begin{align*}
    L(\pi; \phi, \psi)  = \int_{X}\phi(x)d\mu_{0}(x)&+\int_{X}\psi(y)d\mu_{1}(y)\\
            &+\int_{\Omega}(c(\gamma)-(\phi(\gamma(0))+\psi(\gamma(1))))\;d\pi(\gamma).
       \end{align*}

   Minimizing over all $\pi \in \mathcal{M}(\Omega)$, we get:
   
   \begin{proposition}\label{prop: dual problem} Given any pair of functions $\phi\in L^{1}(\mu_{0})$, $\psi\in L^{1}(\mu_{1})$, we define 
    \begin{align*}
       D(\phi, \psi):=\inf_{\pi\geq 0} L(\pi; \phi, \psi).
   \end{align*}
   Then we have
   \begin{align*}
      D(\phi,\psi) = \left \{ \begin{array}{cc}
        \int_{X}\phi(x)d\mu_{0}(x)+\int_{X}\psi(y)d\mu_{1}(y) & \textnormal{ if } \phi(x)+\psi(y) \leq c_{\textbf{e}}(x,y)\quad \forall x,y\\
        -\infty &\textnormal{otherwise}
      \end{array}\right.
    \end{align*}
   \end{proposition}
   \begin{proof} Since the first two terms in $L(\pi; \phi, \psi)$ are independent of $\pi$,
   \begin{align*}
           D(\phi, \psi)&=\int_{X}\phi(x)d\mu_{0}(x)+\int_{X}\psi(y)d\mu_{1}(y)+\inf_{\pi}\int_{\Omega}(c(\gamma)-(\phi(\gamma(0))+\psi(\gamma(1))))\;d\pi(\gamma).
       \end{align*}
 If there exists some $\gamma_{0}$ for which $c(\gamma_{0})-\phi(\gamma_{0}(0))-\psi(\gamma_{0}(1))<0$, then we can make  the minimum $-\infty$. Just take a Dirac mass at $\gamma_{0}$ with very large mass
 \begin{align*}
\pi=\lambda\delta_{\gamma_{0}}\quad\text{for}\quad \lambda>0;
 \end{align*} 
letting $\lambda \to +\infty$,   $D(\phi, \psi)=-\infty$. Else, for all $\gamma \in \Omega$, $c(\gamma)-\phi(\gamma(0))-\psi(\gamma(1))\geq 0$, therefore the third term is $\geq 0$ for all $\pi$, and taking $\pi=0$ we get the inf equal to zero, $D(\phi, \psi)=\int_{X}\phi\;d\mu_{1}+\int_{X}\psi\;d\mu_{1}$.
 \end{proof}
 
 Thus we arrive at the dual problem:
 \begin{definition}\label{ineq: dual}
 (Dual problem.) Let $(\phi, \psi)\in L^{1}(\mu_{0})\times L^{1}(\mu_{1})$.  Let 
 \begin{align*}
 \Pi^{*}:=\{\phi, \psi: \phi(x)+\psi(y)\leq c_{\textbf{e}}(x,y)\}
 \end{align*}
  The dual problem consists of finding the following supremum
\begin{align*}
  \sup\left\{\int_{X}\phi(x)d\mu_{0}(x)+\int_{X}\psi(y)d\mu_{1}(y): (\phi,\psi)\in \Pi^{*}\right\}.
\end{align*}
  \end{definition}
The latter is reminiscent of a dual problem coming from the standard theory of optimal transportation \cite[Section 1.3]{AmbGig2013}. The following is the main fact about the dual problem, we recall $X$ is the closure of  bounded, open, and connected set and we recall the function $c_{\textbf{e}}(x, y)$ defined by (\ref{effective infimum}).
  \begin{theorem}\label{theorem:duality}
  Let $c:\Omega\to \mathbb{R}$ be lower semi-continuous. Let $\mu_{0}\in \mathcal{P}(X)$ and $\mu_{1}\in \mathcal{P}(X)$. Assume that $c_{\textbf{e}}(x,y)\leq f(x)+g(y)$ for some $f\in L^{1}(d\mu_{0}), g\in L^{1}(d\mu_{1})$, then
  \begin{align}
          \inf_{\pi\in \Pi_{\text{path}}}\left\{\int_{\Omega}c(\gamma)\;d\pi(\gamma)\right\}=\sup_{(\phi,\psi)\in \Pi^{*}}\left\{\int_{X}\phi(x)\;d\mu_{0}(x)+\int_{X}\psi(y)d\mu_{1}(y)\right\}
      \end{align}
  \end{theorem}
  The proof of this theorem will be given in the end of Section \ref{Potentials}. For now we go over some important consequences of Theorem \ref{theorem:duality}.

  \begin{corollary}\label{corollary:aux cost}
  Under the assumptions of Theorem \ref{theorem:duality}, if $\pi_{0}\in \Pi_{\text{path}}(\mu_{0},\mu_{1})$ achieves the infimum, then for $\pi_{0}$-a.e. $\gamma \in \Omega$ we have $c(\gamma)=c_{\textbf{e}}(\gamma(0), \gamma(1))$.
  \end{corollary}
  
  \begin{proof}
  Let $(\phi_{0}, \psi_{0})\in \Pi^{*}$ be optimal. Then by Theorem \ref{theorem:duality} we have 
  \begin{align*}
  \int_{\Omega}c(\gamma)\;d\pi_{0}(\gamma)&=\int_{X}\phi_{0}(x)\;d\mu_{0}(x)+\int_{X}\psi_{0}(y)\;d\mu_{1}(y)\\
  &= \int_{\Omega}(\phi_{0}(\gamma(0))+\psi_{0}(\gamma(1)))\;d\pi_{0}(\gamma).
      \end{align*}
  Therefore, 
  \begin{align*}\int_{\Omega}(c(\gamma)-(\phi_{0}(\gamma(0))-\psi_{0}(\gamma(1))))\;d\pi_{0}(\gamma)=0.
  \end{align*}
  Since the integrand is $\geq 0$, not just $\pi_{0}$-a.e. it follows that $c(\gamma)=\phi_{0}(\gamma(0))+\psi_{0}(\gamma(1))$ $\pi_{0}$-a.e. $\gamma$. Since $\phi_{0}(x)+\psi_{0}(y)\leq c_{\textbf{e}}(x, y)$,  $c(\gamma)=\inf\{c(\gamma):\;\gamma(0)=x, \gamma(1)=y\}$ for $\pi_{0}$-a.e. $\gamma$.
  \end{proof}
  
  Corollary \ref{corollary:aux cost} indicates that whenever $\pi$ is optimal, the support of $\pi$ lies in the set of all minimal paths, $\Omega_{\text{min}}$. An explicit proof of this statement, that does not rely on Theorem \ref{theorem:duality}, is given in Lemma \ref{lemma on minimal paths}.

\subsection{Potentials and cyclically monotone sets}  \label{Potentials}
 In this section we further study the properties from the inequality of Definition \ref{ineq: dual}. In some sense, we follow the standard theory of using concave potentials $\phi, \psi$ as in Ambrosio's and Gigli's guide to optimal transport in \cite{AmbGig2013} associated to Theorem \ref{theorem:duality} coming from the theory of superdifferentiability. Namely, a function $\phi$ is called $c$-\textit{concave} if $\phi=\phi^{*}$ as given in (\ref{ineq: inf psi aux}).

Let us write and study such properties. For $\phi\in L^{1}(d\mu_{0})$ and $\psi\in L^{1}(d\mu_{1})$, and for $\mu_{0}$-a.e. $x$ and $\mu_{1}$-a.e $y$ for which $\gamma\in \Omega$ is a path from $\gamma(0)=x$ to $\gamma(1)=y$, we have the inequality from Definition \ref{ineq: dual}. Then the ``concavity" \textit{transforms} stemming from \cite{villani2003topics} are the following. Given $(\phi, \psi) \in \Pi^{*}$, for $\mu_{1}$-a.e. $y$
\begin{align*}
    \psi(y)\leq c_{\textbf{e}}(x,y)-\phi(x).
    \end{align*}
  Taking the infimum with respect $x$,
  \begin{align}\label{ineq: inf psi aux}
          \psi(y)\leq \inf_{x\in X}[c_{\textbf{e}}(x,y)-\phi(x)]:=\phi^{*}(y)
      \end{align}
  Similarly, for $\mu_{0}$-a.e. $x$,
    \begin{align*}
          \phi(x)\leq c_{\textbf{e}}(x,y)-\psi(y),
      \end{align*}
  implementing the infimum with respect to $y$, 
  \begin{align}\label{ineq: inf phi aux}
          \phi(x)\leq \inf_{y \in X}[c_{\textbf{e}}(x,y)-\psi(y)]:=\psi^{*}(x).
      \end{align}
Calling $\mathcal{J}(\phi, \psi):=\int_{X}\phi(x)\;d\mu_{0}(x)+\int_{X}\psi(y)\;d\mu_{1}(y)$, the linear functional from the supremum in Theorem \ref{theorem:duality}, we can witness from (\ref{ineq: inf psi aux}) that
          
  \begin{align*}
          \mathcal{J}(\phi,\phi^{*})\geq \mathcal{J}(\phi, \psi).
      \end{align*}
  And from (\ref{ineq: inf psi aux}) and (\ref{ineq: inf phi aux}) for $\mu_{0}$-a.e. $x$
  
  \begin{align*}
      \phi^{**}(x)=:\inf_{y\in X}[c_{\textbf{e}}(x,y)-\phi^{*}(y)]\geq \phi(x),
  \end{align*}
  and 
  \begin{align*}
          \mathcal{J}(\phi^{**}, \phi^{*})\geq \mathcal{J}(\phi,\phi^{*})\geq \mathcal{J}(\phi, \psi).
      \end{align*}
This heuristic shows that the pair $(\phi, \phi^{*})$ maximizes the dual problem of Definition \ref{ineq: dual}. 

The \textit{superdifferential} set defined for a $c$-concave function $\phi$ is:
 \begin{align*}
          \partial\phi:=\left\{\gamma \in \Omega:\; \phi(x)+\phi^{*}(y)=c_{\textbf{e}}(x,y)\right\}.
      \end{align*}
      
We now give an account on the theory of \emph{cyclical monotone} sets that includes paths. For the classical definition please see Villani's, Santambrogio's, or Ambrosio's and Gigli's account on the theory in \cite{villani2003topics}, \cite{santambrogio2015optimal}, \cite{AmbGig2013}.

Let $\{(x_{i}, y_{i})\}_{i=1}^{N}$ be a set of pair of points in $X$ such that each $x_{i}$ is contained in the support of $\mu_{0}$ and each $y_{i}$ is contained in the support of $\mu_{1}$. Then each path $\gamma_{i}\in \Omega$ in the support of $\pi$ is such that 
 \begin{align*}
     \gamma_{i}(0)=x_{i}\quad \forall i \quad \text{and}\quad
     \gamma_{i}(1)=y_{i}\quad\forall i.
 \end{align*}
 
  The \emph{shift}, as we will call it for the moment, of the final points of paths give rise to new shifted paths $\widetilde{\gamma}$ defined by
  \begin{align}\label{cyclical paths}
      \widetilde{\gamma}_{i}(t)=\gamma_{i}(t)+th_{i}(t),\quad 0\leq t \leq 1, \;h_{i}\ll 1\; \forall i,\quad \text{and}\quad \gamma_{i}(1)+h_{i}(1):=\gamma_{i+1}(1),
  \end{align}
with the convention $\gamma_{N+1}(1)=\gamma_{1}(1)$.  
 The shifted paths  $\{\gamma_{i}\}_{i=1}^{N}$ with endpoints set $\{(x_{i}, y_{i})\}_{i=1}^{N}$ are such that (\ref{cyclical paths}) holds and 
 \begin{align*}
     \widetilde{\gamma}_{i}(0)=x_{i}\quad \forall i\quad \text{and}\quad \widetilde{\gamma}_{i}(1)=y_{i+1\;\text{mod}\;N}\quad \forall i
 \end{align*}
with the convention $\gamma_{N+1}(1)=\gamma_{1}(1)$.
\begin{definition}\label{definition of cyclical sets}
We say the set $\{(x_{i},y_{i})\}$ is \emph{$c_{\textbf{e}}$-cyclically monotone} if for all $i=1,\ldots, n$,
 \begin{align*}
     \sum_{i=1}^{n}c_{\textbf{e}}(x_{i}, y_{i})\leq \sum_{i=1}^{n}c_{\textbf{e}}(x_{i}, y_{\tau(i)}),
 \end{align*}
 for any permutation $\tau$ on $n$ letters.
\end{definition}

One last definition of equal importance is that of minimal paths.
\begin{definition}\label{minimal cyclic path}
A continuous path $\gamma: [0, 1] \to X$ is a \emph{minimal path} from its initial point $x$ to its final point $y$ if for all other paths $\gamma^{\prime}: [0, 1]\to X$ having the same initial and final points of $\gamma$ satisfy
\begin{align*}
    c(\gamma)\leq c(\gamma^{\prime}).
\end{align*}
\end{definition}
For a Borel set $B$, $\pi\llcorner B$ is the restriction of $\pi$ to $B$, namely the measure defined by 
\begin{align*}
    \pi\llcorner B(A)=\pi(B\cap A),\quad\text{for every Borel set}\; A.
\end{align*}

\begin{lemma}\label{lemma on minimal paths}
Let $c: \Omega \to \mathbb{R}$ be a lower semicontinuous cost function. Suppose $\pi$ is optimal. Then if $\gamma_{0}\in\text{spt}(\pi)$, $\gamma_{0}$ is minimal.
\end{lemma}

\begin{proof} 
Suppose $\gamma_{0}\in \text{spt}\;(\pi)$ and $\pi\in \Pi_{\textit{path}}(\mu_{0}, \mu_{1})$ is optimal. Let $\gamma^{\prime}$ be the minimal path from $x_0:=\gamma^{\prime}(0)$ to $y_0 := \gamma^{\prime}(1)$, arguing by contradiction, if $\gamma_{0}$ is not minimal, there is $\varepsilon>0$ such that
\begin{align*}
    c(\gamma^{\prime})<c(\gamma_{0})-\varepsilon/2.
\end{align*}
Let $W:=B_{r}(x_{0})\times B_{r}(y_{0})$ for $r>0$ (to be specified later), $(e_{0},e_{1}):\Omega\to X \times X$ by $\gamma \mapsto (\gamma(0), \gamma(1))$, and define what we will call a \textit{tubo}
\begin{align*}
    T:=(e_{0},e_{1})^{-1}\left(W\right).
\end{align*}
This is an open set in $\Omega$ containing the path $\gamma_{0}$.
Consider the measure, $\bar{\pi}:=\pi\llcorner T/\pi\left[T\right]$; $\pi\left[ T\right]$ will be positive as $\gamma_{0}\in \text{spt}\;(\pi).$ Take $0<\epsilon_{0}<\pi\left[ T\right]$. Define the measures
\begin{align*}
    \nu_{x}:=(e_{0})_{\sharp}\bar{\pi}\quad \text{and}\quad \nu_{y}:=(e_{1})_{\sharp}\bar{\pi}.
\end{align*}
Now build a measure $\widetilde{\pi}\in\Pi(\nu_{x},\nu_{y})$ as follows. Let $g: X\times X \to \Omega$ be defined by 
\begin{align*}
    g_{t}(x,y)=\gamma_{x,y}(t)
\end{align*}
such that $\gamma_{x, y}$ 
is the minimal path from $x$ to $y$. Then set $\widetilde{\pi}:=(g)_{\sharp}\left(\nu_{x}\otimes \nu_{y}\right)$ and define
\begin{align*}
    \pi^{\prime}:=\pi-\epsilon_{0}\bar{\pi}+\epsilon_{0}\widetilde{\pi}.
\end{align*}
That $\pi^{\prime}$ is positive follows from $\pi-\epsilon_{0}\bar{\pi}$ being positive, which is thanks to  $\epsilon_{0}<\pi\left[T\right]$. The marginals of $\pi^{\prime}$ share the same marginals of $\pi$. Concretely, 
\begin{align*}
    (e_{0})_{\sharp}\pi^{\prime}=\mu_{0}-\epsilon_{0}\nu_{x}+\epsilon_{0}(e_{0})_{\sharp}\widetilde{\pi},
\end{align*}
and for every Borel subset $B\subset X$, $(\nu_{x}\otimes \nu_{y})\left[g^{-1}\circ e_{0}^{-1}(B)\right]=\nu_{x}\left[B\right]$. Apply the same argument to get second marginal.

Finally, we will show $\int c(\gamma) d\pi(\gamma)-\int c(\gamma)d\pi^{\prime}(\gamma)>0$, thereby contradicting the optimality of $\pi$. Let $m: \Omega \to \Omega_{\text{min}},\; \gamma\mapsto \gamma_{\text{min}}$ be defined as the minimal $\gamma$ between $x:=\gamma_{\text{min}}(0)=\gamma(0)$ and $y:=\gamma_{\text{min}}(1)=\gamma(1)$. That the map $m$ is well-defined follows from Proposition \ref{boundary value problem}: it tells us that there exists a unique minimal path $\gamma$ between $x$ and $y$. So 
\begin{align*}
    c\circ m(\gamma)=c(\gamma_{\text{min}})\leq c(\gamma).
\end{align*}
Fix $\delta>0$ such that if $\gamma \in T$ ($r>0$ small), then
\begin{align*}
    \|\gamma_{\text{min}}-\gamma^{\prime}\|_{\infty}<\delta\; \;(\mbox{if } r<r(\delta):=\delta/2).
\end{align*}
The lower semi-continuity of $c$ says $c(\gamma_{\text{min}})\geq c(\gamma^{\prime})+\varepsilon/2$. On the other hand, for all $\gamma\in \;\text{spt}(\widetilde{\pi})$, the continuity of $g$ and lower semi continuity of $c$ provides the estimate
\begin{align*}
    c(\gamma)\leq c(\gamma^{\prime})+\varepsilon/4.
\end{align*}
Putting this all together yields,
\begin{align*}
    \int c(\gamma) d\pi(\gamma)-\int c(\gamma)d\pi^{\prime}(\gamma)&=\epsilon_{0}\int c(\gamma)\;d\bar{\pi}(\gamma)-\epsilon_{0}\int c(\gamma)\;d\widetilde{\pi}(\gamma)\\
    &\geq \epsilon_{0}\int c(\gamma_{\text{min}})\;d\overline{\pi}(\gamma)-\epsilon_{0}\int c(\gamma)\;d\widetilde{\pi}(\gamma)\\
    &\geq \epsilon_{0}\int \left(c(\gamma^{\prime})+\frac{\varepsilon}{2}\right)\;d\overline{\pi}(\gamma)-\epsilon_{0}\int \left(c(\gamma^{\prime})+\frac{\varepsilon}{4}\right)\;d\widetilde{\pi}(\gamma)\\
    &=\epsilon_{0}\;\frac{\varepsilon}{4}>0.
\end{align*}

By definition $\overline{\pi}[\Omega]=(\pi\llcorner T)[\Omega]/\pi[T]=1$. Furthermore, we used the fact that both $\widetilde{\pi}$  and $\bar{\pi}$ are supported on $T$ and have unit mass due to the rescaling.
\end{proof}
Next we show that optimal plans in $\Omega$ have cyclical monotone support. The picture complementing the proof of Lemma \ref{c-cyclical monotone} is given below:
\begin{center}
\centering
\includegraphics[height=.45in, width=2.75in]{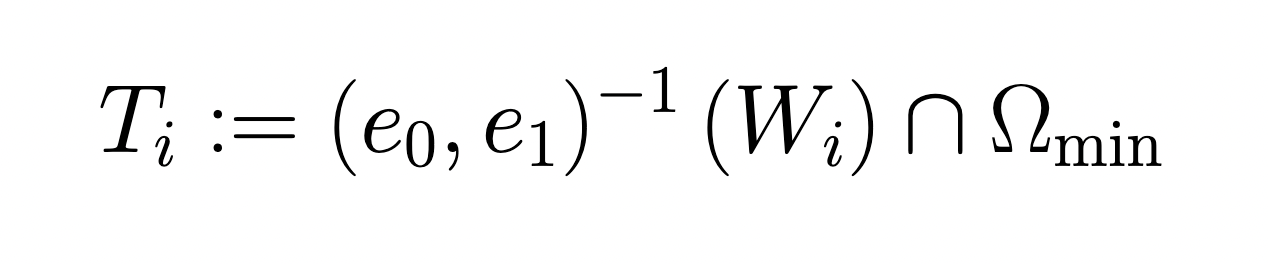}
\includegraphics[height=6.15in, width=6.45in]{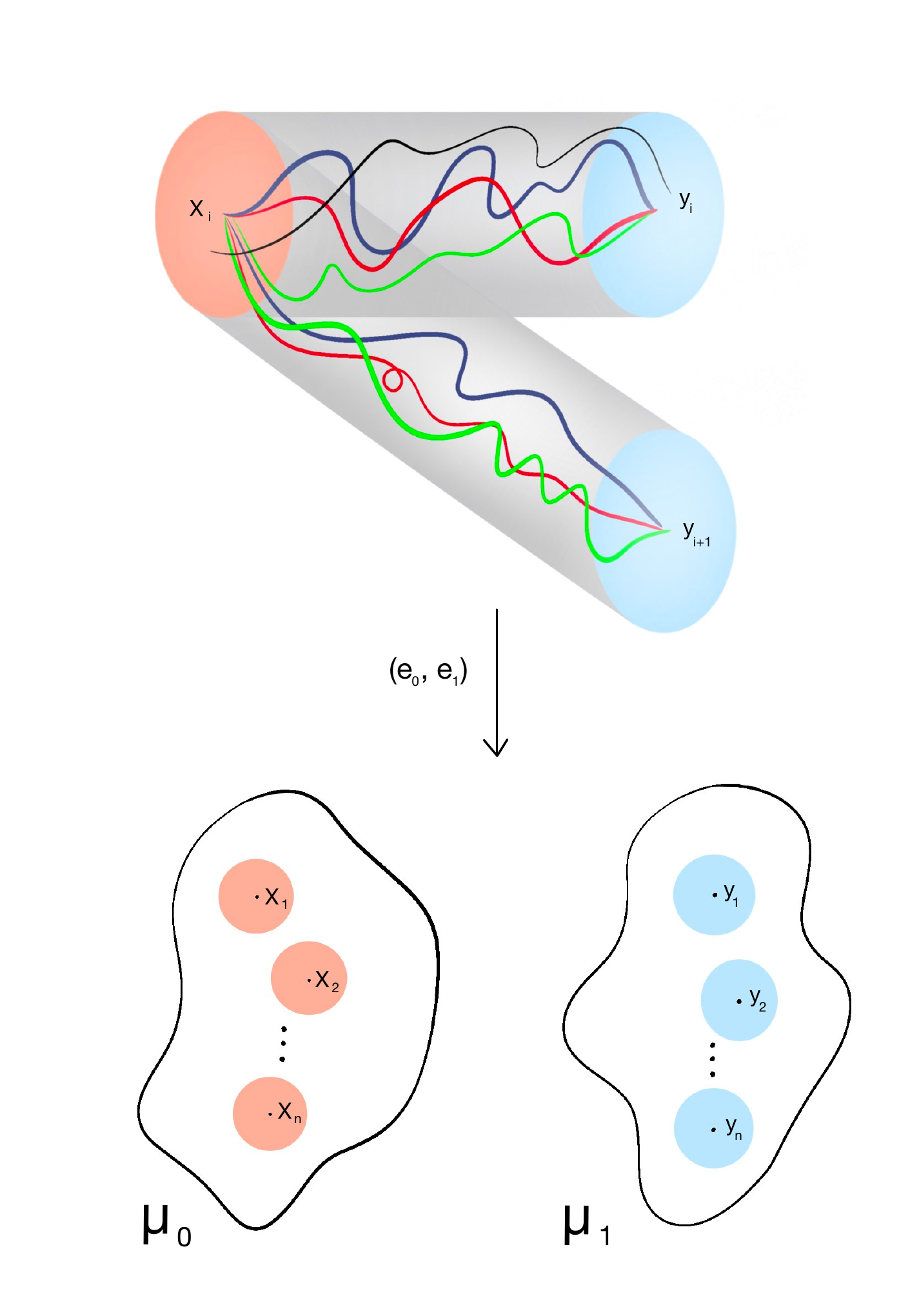}\\
\includegraphics[height=.40in, width=2.75in]{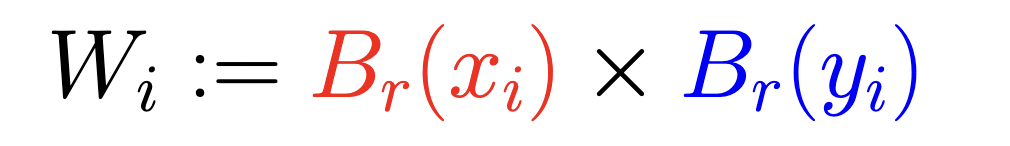}
\end{center}
 \textbf{Figure 1}: This picture indicates how the pull-back mapping $T_{i}$ (the above \textit{tubo}) is obtained from $W_{i}$.

\begin{lemma}\label{c-cyclical monotone}
Suppose $c: \Omega \to \mathbb{R}$ is a continuous cost function and $\pi\in \Pi_{\text{path}}(\mu_{0}, \mu_{1})$ optimal with respect to $c$. Let $\check{\pi}:=(e_{0},e_{1})_{\sharp}\pi \in \Pi(\mu_{0},\mu_{1})$. Then the support of $\check{\pi}$, spt$\;(\check{\pi})$, is $c_{\textbf{e}}$-cyclically monotone. Moreover, $\check{\pi}$ is optimal with respect to $c_{e}(x, y)$ defined by (\ref{effective infimum}).
\end{lemma}

\begin{proof}
We follow a classical proof from standard optimal transportation, in particular see the presentation in Santambrogio's book \cite[Theorem 1.38]{santambrogio2015optimal}. The new concept in the proof is to incorporate path dependence. 

Let $\check{\pi}:=(e_{0}, e_{1})_{\sharp}\pi$ be a transport plan from $\mu_{0}$ onto $\mu_{1}$ obtained by pushing forward $\pi$ through the coupled evaluation map $(e_{0}, e_{1}): \Omega\to X \times X$. Suppose $\pi \in \Pi_{\text{path}}(\mu_{0},\mu_{1})$ is optimal such that $\text{spt}\left(\pi\right)\subset \Omega_{\text{min}}$, the set of all minimal paths by Lemma \ref{lemma on minimal paths}. Suppose by way of contradiction that $\text{spt}\;(\check{\pi})$ is not $c_{e}$-cyclically monotone. Then there exist $n$, $\tau$, and minimal paths $\gamma_{i}$ from  $x_{i}$ to $y_{i}$, and $\widetilde{\gamma}_{i}$ from $x_{i}$ to $y_{\tau(i)}$ in $\text{spt}\;\pi$, respectively $\{(x_{i}, y_{i})\}\subset \text{spt}\;\check{\pi}$, such that 
\begin{align*}
    \sum_{i=1}^{n}c(\gamma_{i})>\sum_{i=1}^{n}c(\widetilde{\gamma}_{i}),
\end{align*}
where the path $\widetilde{\gamma}_{i}$ satisfies (\ref{cyclical paths}) or $\widetilde{\gamma}_{i}(0)=\gamma_{i}(0)$ and $\widetilde{\gamma}_{i}(1)=\gamma_{\tau(i)}(1)$ for all $i=1,\ldots, n$. Then since $\gamma_{i}$ and $\widetilde{\gamma}_{i}$ are minimal, $c(\gamma_{i})=c_{e}(x_{i},y_{i})$ and $c(\widetilde{\gamma}_{i})=c_{e}(x_{i}, y_{\tau(i)})$, and thus the above inequality equals
\begin{align*}
    \sum^{n}_{i=1}c_{e}(x_{i},y_{i})>\sum_{i=1}^{n}c_{e}(x_{i},y_{\tau(i)}).
\end{align*}
Given $\varepsilon>0$, take 
\begin{align}\label{continuous function c}
    \varepsilon<\frac{1}{2n}\left(\sum_{i=1}^{n}c_{e}(x_{i},y_{i})-c_{e}(x_{i},y_{\tau(i)})\right).
\end{align}
By continuity of $c$, there exists an open neighborhood called a \emph{tubo}
\begin{align*}
    T_{i}:=(e_{0},e_{1})^{-1}\left(B_{r}(x_{i})\times B_{r}(y_{i})\right)\cap \Omega_{\text{min}}
\end{align*}
such  that for all $i=1,\ldots, n$ and all $\gamma \in T_{i}$, $c(\gamma)>c(\gamma_{i})-\varepsilon$ and for all $\gamma$ in 
\begin{align*}
    \widetilde{T}_{i}:=(e_{0},e_{1})^{-1}\left(B_{r}(x_{i})\times B_{r}(y_{\tau(i)})\right)\cap \Omega_{\text{min}},
\end{align*}
we have $c(\gamma)<c(\widetilde{\gamma}_{i})+\varepsilon$. 

Now define the measures 
\begin{align*}
    \pi_{i}:=\pi \llcorner T_{i}/\pi\left[T_{i}\right],\quad \nu_{x,i}:=(e_{0})_{\sharp}\pi_{i},\quad\text{and}\quad \nu_{y,i}:=(e_{1})_{\sharp}\pi_{i},
\end{align*}
and note that  $\pi\left[T_{i}\right]$ will be positive for each $i$ since $\gamma_{i}$ is contained in the support of $\pi$. Equivalently $(x_{i},y_{i})\in \text{spt}\;(\check{\pi})$. Take $0<\epsilon_{0}<\min_{i}\pi\left[T_{i}\right]$.

Construct a measure $\widetilde{\pi}_{i}\in \Pi\left(\nu_{x,i},\nu_{y, \tau(i)}\right)$, for every $i$, in the following way. Let $g$ be a map
\begin{align*}
    g: X\times X\to \Omega\quad\text{defined by}\quad (x, y)\mapsto \gamma_{x,y},
\end{align*}
that is, $g(x,y)=\gamma_{x,y}(t)$ the minimal path between $\gamma_{x,y}(0)=x$ and $\gamma_{x,y}(1)=y$. Then the estimates $c(\gamma)>c(\gamma_{i})-\varepsilon$ for all $\gamma \in T_{i}$ and $c(\gamma)<c(\tilde{\gamma}_{i})+\varepsilon$ for all $\gamma \in \widetilde{T}_{i}$ coincide with the estimates $c_{e}(x,y)>c_{e}(x_{i}, y_{i})-\varepsilon$ and $c_{e}(x,y)<c_{e}(x_{i},y_{\tau(i)})+\varepsilon$ for all $(x, y)\in B_{r}(x_{i})\times B_{r}(y_{i})$ and all $(x, y)\in B_{r}(x_{i})\times B_{r}(y_{\tau(i)})$, respectively. Take $\widetilde{\pi}_{i}:=(g)_{\sharp}\left(\nu_{x,i}\otimes \nu_{y,\tau(i)}\right)$.

Now define
\begin{align*}
    \widetilde{\pi}:=\pi-\epsilon_{0}\sum_{i=1}^{n}\pi_{i}+\epsilon_{0}\sum_{i=1}^{n}\widetilde{\pi}_{i}.
\end{align*}
That $\widetilde{\pi}$ is positive follows from $\pi-\epsilon_{0}\sum_{i=1}^{n}\pi_{i}$ being positive as $\epsilon_{0}<\min_{i}\pi\left[T_{i}\right]$. More concretely, since it suffices to check $\pi-\epsilon_{0}\sum_{i=1}^{n}\pi_{i}>0$, the condition $\epsilon_{0}\pi_{i}<\pi/n$ is enough. Indeed, as $\epsilon_{0}\pi_{i}=\frac{\epsilon_{0}}{\pi\left[T_{i}\right]}\pi\llcorner T_{i}$ and $\epsilon_{0}/\pi\left[T_{i}\right]\leq 1/n$.

The marginals of $\widetilde{\pi}$ share the marginals of $\pi$:
\begin{align*}
    (e_{0})_{\sharp}\widetilde{\pi}=\mu_{0}-\epsilon_{0}\sum_{i=1}^{n}\nu_{x,i}+\epsilon_{0}\sum_{i=1}^{n}(e_{0})_{\sharp}(g)_{\sharp}\left(\nu_{x,i}\otimes \nu_{y,\tau(i)}\right),
\end{align*}
and for all Borel subsets $B\subset X$, $\left(\nu_{x,i}\otimes \nu_{y, \tau(i)}\right)\left[g^{-1}\left(e_{0}^{-1}(B)\right)\right]=\nu_{x,i}[B]$ is the measure $\nu_{x,i}$ containing the points over the entire first copy of $X$. The second marginal follows the same story:
\begin{align*}
    (e_{1})_{\sharp}\widetilde{\pi}=\mu_{1}-\epsilon_{0}\sum_{i=1}^{n}\nu_{y,i}+\epsilon_{0}\sum_{i=1}^{n}(e_{1})_{\sharp}(g)_{\sharp}\left(\nu_{x,i}\otimes \nu_{y,\tau(i)}\right);
\end{align*}
for all Borel $A\subset X$, $\left(\nu_{x,i}\otimes \nu_{y, \tau(i)}\right)\left[g^{-1}\left(e_{1}^{-1}(A)\right)\right]=\nu_{y,\tau(i)}[A]$ is the measure $\nu_{y,\tau(i)}$ containing the points over the entire second copy of $X$.

Finally, the estimate $\int c\;d\pi-\int c\;d\widetilde{\pi}$ is positive, thereby contradicting the optimality of $\pi$:
\begin{align*}
    \int c(\gamma)\;d\pi(\gamma)-\int c(\gamma)\;d\widetilde{\pi}(\gamma)&= \epsilon_{0}\sum_{i=1}^{n}\int c(\gamma)\;d\pi_{i}(\gamma)-\epsilon_{0}\sum_{i=1}^{n}\int c(\gamma)\;d\widetilde{\pi}_{i}(\gamma)\\
    &\geq \epsilon_{0}\sum_{i=1}^{n}\left(c(\gamma_{i})-\varepsilon\right)-\epsilon_{0}\sum_{i=1}^{n}\left(c(\widetilde{\gamma}_{i})+\varepsilon\right)\\
    &=\epsilon_{0}\left(\sum_{i=1}^{n}c_{e}(x_{i},y_{i})-c_{e}(x_{i},y_{\tau(i)})-2n\varepsilon\right)>0,
\end{align*}
where we used that $\pi_{i}$ is supported on $T_{i}$, $\widetilde{\pi}_{i}$ supported on $\widetilde{T}_{i}$ and have unit mass by rescaling the measures by $\pi\left[T_{i}\right]$.

To end the proof, we apply the standard theory of optimal transportation, from \cite[Theorem 1.13]{AmbGig2013}, to the endpoints $(x_{i}, y_{i})$ contained in the support of $\check{\pi}$ to find that $\check{\pi}$ is optimal with respect to $c_{e}(x, y)$, since $\text{spt}\;\check{\pi}$ is $c_{e}$-cyclically monotone.
\end{proof}

\begin{remark}
The previous Lemma \ref{c-cyclical monotone} says that any optimal plan $\pi$ in $\Pi_{\text{path}}(\mu_{0},\mu_{1})$ ``projects" to a solution of the Monge-Kantorovich problem with cost $c_{\textbf{e}}(x,y)$ and coupling $\check{\pi}:=(e_{0},e_{1})_{\sharp}\pi$. Recall the notion of dynamical couplings from Section \ref{Section: main results}. In both cases we get the same dynamical optimal coupling, but our proofs are different. 
\end{remark}

Armed with these results we may, and we actually do, prove Theorem \ref{theorem:duality}. 
  
\begin{proof}[\textbf{Proof of Theorem
\ref{theorem:duality}}]
With all the assumptions of Theorems \ref{theorem:duality} and Lemma \ref{c-cyclical monotone}, let $\pi\in\Pi_{\text{path}}(\mu_{0},\mu_{1})$; notice that for any pair $(\phi, \psi)\in L^{1}(d\mu_{0})\times L^{1}(d\mu_{1})$ satisfying inequality of Definition \ref{ineq: dual},
  \[
  \int_{\Omega}c_{\textbf{e}}(x,y)d\pi(\gamma)\geq \int_{\Omega}(\phi(x)+\psi(y))d\pi(\gamma)=\int_{X}\phi(x)d\mu_{0}(x)+\int_{X}\psi(y)d\mu_{1}(y).
  \]
Next, take infimum over admissible  $\pi$ on the left-hand side  and take the supremum over $(\phi, \psi)\in \Pi^{*}$ on the right-hand side to get the ``$\geq$" part.

To prove the reverse inequality, choose an optimal $\pi \in \Pi_{\text{path}}(\mu_{0},\mu_{1})$. Since $\text{spt}\;\check{\pi}$ is $c_{e}$-cyclically monotone by Lemma \ref{c-cyclical monotone}, the classical theory of optimal transport applies to show there is a  lower semi-continuous concave function $\phi$ such that $\text{spt}\;\check{\pi}\subset \partial\phi$ for which $\phi\in L^{1}(d\mu_{0})$ and $\phi^{*}\in L^{1}(d\mu_{1})$. Then
\begin{align}\label{eqn: integral aux}
      \int_{\Omega}c_{\textbf{e}}(x,y)d\pi(\gamma)&=\int_{\Omega}(\phi(x)+\phi^{*}(y))d\pi(\gamma)=\int_{X}\phi(x)d\mu_{0}(x)+\int_{X}\phi^{*}(y)d\mu_{1}(y)
      \end{align}
  Now, the claim is then that $(\phi, \phi^{*})$ solves the maximization problem of Definition \ref{ineq: dual}. More concretely,
  to prove that $(\phi,\phi^{*})$ solves the maximization problem, we note that\\
  1. $\tilde{\phi}(x)+\tilde{\phi}^{*}(y)=c_{\textbf{e}}(x,y)$ on the support of the optimal $\check{\pi}$, spt$\check{\pi}$.\\
  2. $\phi(x)+\phi^{*}(y)\leq c_{\textbf{e}}(x,y)$ on $X\times X$.
  Then
  \begin{align*}
          \int_{X}\tilde{\phi}(x)d\mu_{0}(x)+\int_{X}\tilde{\phi}^{*}(y)d\mu_{1}(y)&=\int_{\Omega}(\tilde{\phi}(x)+\tilde{\phi}^{*}(y))d\pi(\gamma)\\
          &=\int_{\Omega}c_{\textbf{e}}(x,y)\;d\pi(\gamma)\\
          &\geq \int_{\Omega}(\phi(x)+\phi^{*}(y))d\pi(\gamma)\\
          &=\int_{X}\phi(x)d\mu_{0}(x)+\int_{X}\phi^{*}(y)d\mu_{1}(y),
      \end{align*}
  and so $\tilde{\phi}$ solves the maximization problem. This with (\ref{eqn: integral aux}) establishes the proof.
  \end{proof}
  
  \subsection{Optimal plans given by maps } \label{Optimal plans given by maps}
  The goal of this section is to prove  uniqueness of the minimizer of the optimal path  Kantorovich problem (\ref{costpath}) and also show it is given by a map $\Gamma(x)$, provided $\mu_{0}$ is absolutely continuous with respect to the Lebesgue measure: $\mu_{0}(x)<< dx$. This uses, and extends, results of Brenier \cite{Brenier1991} and Gangbo-McCann \cite{gangbo1996geometry} in the classical optimal transport theory.
  
Let $\Omega_{\text{min}}$ denote the set of all minimal paths. If $\pi\in \Pi_{\text{path}}(\mu_{0},\mu_{1})$ is optimal Corollary \ref{corollary:aux cost} indicated $\text{spt}\;\pi$ is contained in $\Omega_{\text{min}}.$ In particular, the calculations and discussions of Section \ref{section:minimal paths} and Section \ref{The dual problem and the endpoint cost function} amount to the existence of a minimal path $\gamma_{*}\in \Omega$, minimizing for $c(\gamma)=\int_{0}^{1}|\dot{\gamma}(t)|^{2}-V(\gamma(t), t)\;dt$ in $\Omega$,
in  the support of $\pi$ contained in $\Omega_{\text{min}}$. 
  
The uniqueness will follow. Moreover, the mapping $\Gamma$ will be given  by (\ref{eqn: gammamap}) below and such that the optimal transport plan will be given by $\pi_{\Gamma}:=(\Gamma)_{\sharp}\mu_{0}$ 
  
  We are going to look for mappings
  \begin{equation}
      \begin{split}
          \Gamma: X \times [0, 1] \to X
      \end{split}
  \end{equation}
of the form $\Gamma(x,t)$ for every $x$ with the following properties
\begin{equation} \label{eqn: gammamap}
\begin{split}
\left\{ \begin{array}{rcl}
\Gamma(x,0)& =x \mbox{ for all }
& x \in X\\ \Gamma(x,1) & =T(x)  \mbox{ for } &  T_{\sharp}\mu_{0}=\mu_{1} 
\end{array}\right.
\end{split} 
\end{equation}
where $T: X \to X$ is a measurable map pushing $\mu_{0} \mapsto \mu_{1}$. 

Once and for all we shall consider functions from $X\times [0,1]$ to $X$ to be in one-to-one correspondence with functions from $X$ to $\Omega$; we shall denote them with the same letter $\Gamma$. 

The next result says that if $T$ above maps $\mu_0$ to $\mu_1$, then $\Gamma$ maps $\mu_0$ to an admissible measure in the path space.

\begin{lemma}\label{pi is admissible}
Define $\pi_{\Gamma}:=(\Gamma)_{\sharp}\mu_{0}.$ Then $\pi_{\Gamma}\in \Pi_{\text{path}}(\mu_{0},\mu_{1})$.
\end{lemma}

\begin{proof} 
In a moment we will unpack some definitions regarding push-forwards thru evaluation maps  (see Bernot et al in \cite[Section 3]{bernot2008optimal} for an elementary introduction).  First recall the evaluation maps  $e_{0},e_{1}: \Omega \to X$ are such that $e_{0}(\gamma)=\gamma(0)$ and $e_{1}(\gamma)=\gamma(1)$. For measurable subsets $A,B \subset X$ and $\pi \in \mathcal{P}(\Omega)$ we have
\[
(e_{0})_{\sharp}\pi_{\Gamma}[A]= \pi_{\Gamma}[e_{0}^{-1}(A)]=\pi_{\Gamma}[\{\gamma \in \Omega: \gamma(0)\in A\}]
\]
\[
(e_{1})_{\sharp}\pi_{\Gamma}[B]= \pi_{\Gamma}[e_{1}^{-1}(B)]=\pi_{\Gamma}[\{\gamma \in \Omega: \gamma(1)\in B\}]
\]

To show $(\Gamma)_{\sharp}(\mu_{0})$ lies in $\Pi_{\text{path}}(\mu_{0},\mu_{1})$, we must show the following two things. The first $(e_{0})_{\sharp}(\Gamma)_{\sharp}(\mu_{0})=\mu_{0}$ and the second $(e_{1})_{\sharp}(\Gamma)_{\sharp}(\mu_{0})=\mu_{1}$. Notice that $T: X \to X$ pushes forward $\mu_{0}$ to $\mu_{1}$. Then, observe that the composition $e_{1}\circ \Gamma$ coincides with $T$, and since $T_{\sharp}\mu_{0}=\mu_{1}$, $(e_{1} \circ \Gamma)_{\sharp}\mu_{0}=\mu_{1}$.

Let $A,B \subset X$ be measurable subsets.  Then $\pi_{\Gamma}[A]:=(\Gamma)_{\sharp}(\mu_{0})[A]=\mu_{0}[\Gamma^{-1}(A)]$ so that
 \begin{align*}
(e_{0})_{\sharp}\pi_{\Gamma}[A]
&=\pi_{\Gamma}[e^{-1}_{0}(A)]\\
&=\mu_{0}[\Gamma^{-1}(e^{-1}_{0}(A))]\\
&= \mu_{0}[A].
\end{align*}
 Similarly, we have $\pi_{\Gamma}[B]:=(\Gamma)_{\sharp}(\mu_{0})[B]=\mu_{0}[\Gamma^{-1}(B)]$. Then
 
  \begin{align*} 
 (e_{1})_{\sharp}\pi_{\Gamma}[B]
&=\pi_{\Gamma}[e_{1}^{-1}(B)]\\
&=\mu_{0}[\Gamma^{-1}(e_{1}^{-1}(B))]\\
&=\mu_{1}[B],
\end{align*}
and this completes the proof.
\end{proof}

The important result we use from the standard optimal theory is that whenever $\mu_{0}$ is absolutely continuous with respect to Lebesgue \cite{AmbGig2013}, \cite{villani2003topics}, the optimal plan will be concentrated on the graph of $T$. Furthermore, the classical theory of optimal transportation already contains a theorem for a Lagrangian cost function, see Villani's book \cite[Chapter 5]{villani2003topics}: the following has a unique solution,
  \begin{align}\label{time-dependent ot}
      \inf\left\{\int_{\mathbb{R}^{n}}c_{L}(x,T(x))\;d\mu_{0}(x): \gamma_{0}=\text{id},\; \gamma_{1}=T\;\text{such that}\; T_{\sharp}\mu_{0}=\mu_{1}\right\},
  \end{align}
where $c_{L}(x,y):=\min\left\{\int_{0}^{1}L(\dot{\gamma}(t))\;dt:\gamma(0)=x,\; \gamma(1)=y\right\}$ and $L$ is a strictly convex Lagrangian cost function satisfying $L(0)=0$. The solution enjoys Brenier's characterization \cite{Brenier1991} and it is a consequence of Theorem 10.28 in Villani's book \cite{villani2009optimal}; so it is given by 
\begin{align*}
    \gamma_{t}(x)=x-t\nabla^{*}L(\nabla\psi(x)),\quad 0\leq t\leq 1,
\end{align*}
where $\psi$ is a $c_{L}$-concave function ( see Section \ref{Potentials} ) for which $[\text{id}-\nabla^{*}L(\nabla\psi)]_{\sharp}\mu_{0}=\mu_{1},$ and $\nabla^{*}$ denotes the Legendre transform or in this case can be thought of as the inverse, $\nabla^{-1}L=\nabla^{*}L$. If such solution exists, then the optimizer should interpolate between $\gamma_{0}(x)=x$ and $\gamma_{1}(x)=x-\nabla^{*}L(\nabla\psi(x))$, according to Villani \cite[Chapter 5]{villani2003topics}. Monge's classical minimization Problem (\ref{monge min}) and  Problem (\ref{time-dependent ot}) are compatible provided $c_{L}(x, y)=\min\left\{\int_{0}^{1}L(\dot{\gamma}_{t})dt:\gamma_{0}=x,\; \gamma_{1}=y\right\}$. In this case solutions of the time-dependent minimization optimal transport problem have to satisfy for $\mu_{0}$-a.e. $x$, $c_{L}(x,T(x))=\int_{0}^{1}L(\dot{\gamma}_{t})dt$.  To solve this problem, we took a different approach.

Returning to our main point of view of optimal path  optimal transport theory and possessing this knowledge, we are ready to derive an optimal path  map stemming from the set of minimal paths which will include the mapping (\ref{eqn: gammamap}): $\Gamma: X \to \Omega$.  Let $e_{0}: \Omega\to X$ be the evaluation map and recall $X:=\overline{B_{R}(0)}\subset \mathbb{R}^{n}$.

 \begin{proof}[\textbf{Proof of Theorem \ref{result 2: uniqueness Kantorovich solution on paths}}]
Theorem \ref{result 1: Kantorovich on paths minimizers} says there is an optimal plan $\pi_{*}$ in the space of paths solving the optimal path problem (\ref{costpath}), \textbf{Problem A}; while Lemma \ref{c-cyclical monotone} tells us that  $\pi_{*}$ projects to the classical optimal transport solution with respect to $c_{\textbf{e}}(x, y)$, namely $\check{\pi}_{*}:=(e_{0},e_{1})_{\sharp}\pi_{*}$. 

Lemma \ref{effective cost is differentiable} says $c_{\textbf{e}}(x, y)$ is differentiable and Lemma \ref{twist condition} says it is injective in its domain. Since $\pi_{*}$ is optimal, Lemma \ref{c-cyclical monotone} implies the support of $\check{\pi}_{*}$ is $c_{e}$-cyclically monotone. The classical theory of optimal transportation \cite[Theorem 1.13]{AmbGig2013} says that $\text{spt}\;\check{\pi}_{*}$ is contained in $\partial \phi$, and as $\phi$ is locally Lipschitz apply Rademacher's theorem and that $\mu_{0}\ll dx$, $\phi$ is differentiable $\mu_{0}$-a.e. All of the above show $c_{\textbf{e}}(x, y)$ satisfies the assumptions of Theorem 10.28 in Villani's book \cite{villani2009optimal} and so it applies to give a unique transport map $T$ pushing $\mu_{0}$ forward to $\mu_{1}$, solving the optimal transport problem with respect to $c_{\textbf{e}}(x, y)$. 

Define a map $\Gamma$ as in (\ref{eqn: gammamap}) containing the above data on $T$. Let $\gamma_{x,y}(t)$ be the minimal path between $x$ and $y$. Then $c(\gamma_{x,y}(t))=c_{\textbf{e}}(x,y)$, and using this information, let $\Gamma(x, t):=\gamma_{x, T(x)}(t)$.  More precisely, $\Gamma$ is given as the composition,
\begin{center}
\begin{tikzcd}
 X\ar[r,"(Id\times T)"]\ar[rr,out=-30,in=210,swap,"\Gamma"] & X\times X\ar[r,"\gamma_{x,y}"] & \Omega_{\text{min}}
\end{tikzcd}.
\end{center}

So since $\text{spt}\;\pi_{*}$ lies in $\Omega_{\text{min}}$, $\text{spt}\;\pi_{*}$ is concentrated on the graph of the mapping $\Gamma$.  Therefore, $\Gamma$ uniquely solves the Monge optimal path problem. 
\end{proof}

\begin{remark}
Minimizers are thus given by maps. A nontrivial question to think about then is: are these maps continuous? The regularity of optimal transport maps is an important and active area of research. One should note that  when there is no interaction term, one could apply the standard optimal transport regularity theory of Guillen-Kitagawa \cite{guillen2015local}, Figalli-Kim-McCann \cite{figalli2013holder}, and Ma-Trudinger-Wang \cite{ma2005regularity}  to understand  regularity in the path dependent case. However, in the case one has interaction terms, then  regularity becomes much more difficult and poses a natural and interesting problem. 
\end{remark}

\section{The optimal path problem with interaction}\label{Path MKP interaction}

In this section we consider a added interaction term to the functional $\mathcal{E}_{0}(\pi)$, from Section \ref{Path MKP}.  We recall the linear functional in (\ref{interaction term}):
\begin{align*}
    \mathcal{E}(\pi):=\int_{\Omega}c(\gamma)\;d\pi(\gamma)+\int_{\Omega}\int_{\Omega}\mathcal{K}(\gamma, \sigma)\;d\pi(\sigma)\;d\pi(\gamma),
\end{align*}
where the Kernel $\mathcal{K}$ is given by (\ref{exp path}).
The plan is to determine properties of minimizers of (\ref{interaction term}) over $\Pi_{\text{path}}$ and to determine whether the optimal plans are given by maps. 

\subsection{Existence of minimizers for the optimal path  with interaction}\label{Existence of minimizers for the optimal path  with interaction}
The next result is proving existence of
minimizers for the cost $c(\gamma)$ with added interaction term, namely proving Theorem \ref{result 3: Kantorovich with interaction solutions}. That is, we prove existence of solutions of (\ref{MKpath}), where $\mathcal{U}$ is given by (\ref{interaction term}) with (\ref{exp path}).

  \begin{proof}[\textbf{Proof of Theorem \ref{result 3: Kantorovich with interaction solutions}}]
What will help us achieve the existence of a minimizer is two-fold: Condition \ref{condition:weaker} to help us get enough compactness, just like in the
proof of Theorem \ref{result 1: Kantorovich on paths minimizers} and the next Lemma \ref{convergence of iterated integrals} which will
allow us to use continuity to pass to the limit in the iterated integrals.

Since $\mathcal{K}\geq 0$ we have $\mathcal{E}_{0}(\pi)\leq \mathcal{E}(\pi)$ for all $\pi$ in $\Pi_{\text{path}}(\mu_{0},\mu_{1})$. Thus $\inf_{\pi}\mathcal{E}(\pi)\geq \inf_{\pi}\mathcal{E}_{0}(\pi)\geq 0$; since the infimum is finite, there is a minimizing sequence $\{\pi_{k}\}_{k}$.

 On the other hand using the same compact set, $X_{N}$, from Theorem \ref{result 1: Kantorovich on paths minimizers} and coercivity condition, Condition \ref{condition:weaker}, we deduce
 \begin{align*}
\pi(X_{N}^{c})&=\int_{X_{N}^{c}}\;d\pi(\gamma)=\frac{1}{N}\int_{X_{N}^{c}}N\;d\pi(\gamma)\\
&\leq\frac{1}{N}\int_{\Omega}c(\gamma)\;d\pi(\gamma)\\
&=\frac{1}{N}\mathcal{E}_{0}(\pi)\leq
\frac{1}{N}\mathcal{E}(\pi).
\end{align*}

Then
\[
\mathcal{E}(\pi_{k}) \longrightarrow \inf_{\pi\in
\Pi_{\text{path}}(\mu_{0},\mu_{1})}\mathcal{E}(\pi) \quad \text{as}\quad k
\to \infty.
\] 
Since $\{\pi_{k}\}_{k}$ is a minimizing
sequence, then $\mathcal{E}(\pi_{k})\leq C$ for all $k$. Since $\pi(X_{N})=\pi(\Omega_{M, N})$, we then have
\begin{align*}
 \pi_{k}(\Omega_{M, N}^{c})&\leq \frac{\mathcal{E}(\pi_{k})}{N} \leq\frac{C}{N}\quad
\text{for all }\quad k.
\end{align*}
So given $\varepsilon:=\frac{C}{N}>0$,
$\pi_{k}(\Omega_{M, N}^{c})\leq \frac{C}{N}=\varepsilon \; \text{for
all }\;k.$
This says that the sequence $\{\pi_{k}\}_{k}$ is tight, as $\Omega_{M, N}$ is compact (Condition \ref{condition:weaker}); then Prokhorov's theorem (Theorem \ref{Prokhorov}) tells us that there exists a subsequence $\{\pi_{k_j}\}_{j}$ in
$\Pi_{\text{path}}(\mu_{0},\mu_{1})$ and a Borel probabilty measure
$\varphi$ such that $\pi_{k_{j}} \rightharpoonup \varphi$ as $k \to
\infty$. From the proof of Theorem \ref{result 1: Kantorovich on paths minimizers}, we see
$\varphi$ is an element of $\Pi_{\text{path}}(\mu_{0},\mu_{1})$.

Next we will show
\[
\mathcal{E}(\varphi)\leq \liminf_{j\to \infty}\mathcal{E}(\pi_{k_j}).
\]

Here the proof differs from the non-interaction one. Lemma \ref{convergence of iterated integrals} below; the dominated
convergence theorem give

  \begin{align*}
\int_{\Omega}c(\gamma)\;d\varphi(\gamma)+\int_{\Omega}\int_{\Omega}\mathcal{K}(\gamma, \sigma)d\varphi(\sigma)\;d\varphi(\gamma)&=
\lim_{n \to \infty}\Big(\int_{\Omega}c_{n}(\gamma)d\varphi(\gamma)
+\int_{\Omega}\int_{\Omega}\mathcal{K}_{n}( \gamma, \sigma)d\varphi(\sigma)d\varphi(\gamma)\Big)\\
&= \lim_{n\to \infty}\lim_{j \to
\infty}\int_{\Omega}c_{n}(\gamma)d\pi_{k_j}(\gamma)+\int_{\Omega}\int_{\Omega}\mathcal{K}_{n}(\gamma, \sigma)d\pi_{k_j}(\sigma)d\pi_{k_j}(\gamma)\\
&\leq \liminf_{j \to \infty}\Big(\int_{\Omega}c(\gamma)d\pi_{k_j}(\gamma)
+\int_{\Omega}\int_{\Omega}\mathcal{K}(\gamma, \sigma)d\pi_{k_j}(\sigma)d\pi_{k_j}(\gamma)\Big).
\end{align*}
  We have therefore proved
  \[
    \mathcal{E}(\varphi)\leq \liminf_{j\to \infty}\mathcal{E}(\pi_{k_j}),
\]
and therefore $\varphi$ is indeed a minimizer of $\mathcal{E}(\pi)$, as
wanted to be shown.
\end{proof}

An important property that will be used in the proof of Lemma \ref{convergence of iterated integrals} comes form the following definition.

\begin{definition}\label{modulus of continuity}
Let $X$ be a metric space and let $f: X \to \mathbb{R}$ be a function and let $b: [0, +\infty]
\to [0, +\infty]$ be a function. We say that $f$ has \textit{modulus of continuity} $b$ if for all $x, y \in X$, $|f(x)-f(y)|\leq b(d(x,y)).$
\end{definition}

\begin{lemma}\label{convergence of iterated integrals}
Let $(\pi_{k})_{k}$ be a sequence of  probability measures on
$\Omega$. Suppose $\mathcal{U}(\gamma, \pi)$, defined in (\ref{fancy u}), is bounded  and $\pi_{k}$, $\pi$ have finite total mass. If $\pi_{k}
\rightharpoonup  \pi$ weakly as $k \to \infty$, then
\[
\int_{\Omega}\mathcal{U}(\gamma,\pi_{k})\;d\pi_{k}(\gamma)\longrightarrow
\int_{\Omega}\mathcal{U}(\gamma,\pi)\;d\pi(\gamma)\quad \text{as} \quad k
\longrightarrow \infty.
\]
\end{lemma}

\begin{proof}
Recall the definitions of (\ref{interaction term}) and (\ref{exp path}).
In terms of this notation we will show

\[
\left|\int_{\Omega}\int_{\Omega}\mathcal{K}(\gamma,\sigma)\;d\pi_{k}(\sigma)\;d\pi_{k}(\gamma)-\int_{\Omega}\int_{\Omega}\mathcal{K}(\gamma,\sigma)\;d\pi(\sigma)\;d\pi(\gamma)\right|\longrightarrow
0\quad \text{as}\quad k \to \infty.
\]

For each compact set $\Omega_{M, N}$ arising from the coercivity property with
given $\varepsilon:=\frac{C}{N}$, we shall show
\[
\left|\int_{\Omega}\int_{\Omega}\mathcal{K}(\gamma,\sigma)\;d\pi_{k}(\sigma)\;d\pi_{k}(\gamma)-\int_{\Omega}\int_{\Omega}\mathcal{K}(\gamma,\sigma)\;d\pi(\sigma)\;d\pi(\gamma)\right|
\to 0 \quad \text{as} \quad k \to \infty.
\]

Let us both unpack this.
Indeed we have this expression is equal  to 
\begin{equation}\label{freezing variable 1}
\begin{split}
\Big|\int_{\Omega}\int_{\Omega}\mathcal{K}(\gamma,\sigma)d\pi_{k}(\sigma)d\pi_{k}(\gamma)&-\int_{\Omega}\int_{\Omega}\mathcal{K}(\gamma,\sigma)d\pi(\sigma)d\pi_{k}(\gamma)+\int_{\Omega}\int_{\Omega}\mathcal{K}(\gamma,\sigma)d\pi(\sigma)d\pi_{k}(\gamma)\\
&-\int_{\Omega}\int_{\Omega}\mathcal{K}(\gamma,\sigma)d\pi(\sigma)d\pi(\gamma)\Big|,
\end{split}
\end{equation}
applying the triangle inequality the above (\ref{freezing variable 1}) is
less than or  equal to
\begin{equation}
\begin{split}\label{triangle ineq}
\int_{\Omega}\Big|\int_{\Omega}\mathcal{K}(\gamma,\sigma)d\pi_{k}(\sigma)-\int_{\Omega}\mathcal{K}(\gamma,\sigma)d\pi(\sigma)\Big|d\pi_{k}(\gamma)&+\Big|\int_{\Omega}\int_{\Omega}\mathcal{K}(\gamma,\sigma)d\pi(\sigma)d\pi_{k}(\gamma)\\
&-\int_{\Omega}\int_{\Omega}\mathcal{K}(\gamma,\sigma)d\pi(\sigma)d\pi(\gamma)\Big|\\
:=I_{k}+J_{k}.
\end{split}
\end{equation}

Let us first investigate $J_{k}$. Since $\pi_{k}\rightharpoonup \pi$ as $k$
tends to infinity, we know that
\[
\int_{\Omega}f(\sigma)\;d\pi_{k}(\sigma)\longrightarrow
\int_{\Omega}f(\sigma)\;d\pi(\sigma)\quad  \text{for all }\quad f\in
C_{b}(\Omega).
\]
Fix $\gamma$ and apply this to $f(\sigma):=\mathcal{K}(\gamma, \sigma)$. So
 \[
 \mathcal{U}(\gamma,\pi_{k}) \longrightarrow \mathcal{U}(\gamma, \pi)\quad
\text{as}\quad  k \to \infty.
 \]

\emph{\textbf{Claim.} If $\mathcal{K}(\gamma, \sigma)$ has a $b$-modulus of continuity in the first
coordinate, then $\mathcal{U}(\gamma,\pi)$ has a $b$-modulus of continuity
in the first coordinate independent of $\sigma$}.

Indeed, using the hypothesis that $\mathcal{K}$ has a $b$ modulus of continuity and then integrating in $\sigma$ for any $\pi$ probability measure on $\Omega$, we have
\begin{equation} \label{ineq: uniformity}
\begin{split}
|\mathcal{U}(\gamma_{1}, \pi)-\mathcal{U}( \gamma_{2}, \pi)|&=\Big|\int_{\Omega}\mathcal{K}(\gamma_{1}, \sigma)d\pi(\sigma)-\int_{\Omega}\mathcal{K}(\gamma_{2}, \sigma)d\pi(\sigma)\Big|\\
&\leq \int_{\Omega}\Big|\mathcal{K}(\gamma_{1}, \sigma)-\mathcal{K}(\gamma_{2}, \sigma)\Big|d\pi(\sigma)\\
&\leq b(\|\gamma_{1}-\gamma_{2}\|)
\end{split}
\end{equation}

Then $\mathcal{U}(\gamma,\pi)$ is continuous and bounded, so since $\pi_{k} \rightharpoonup \pi$, we conclude that (recall (\ref{triangle ineq})) $J_{k} \to 0$ as $k\to \infty$.

Next we study $I_{k}$. This one is a bit more delicate. The idea is to break
the integral on the compact set $\Omega_{M, N}$ and outside the compact set and use Arzela-Ascoli.
\begin{equation} \label{eq: integralI0 integralI1}
\begin{split}
 I_{k}&=
\int_{\Omega_{M, N}}\Big|\mathcal{U}(\gamma,\pi_{k})-\mathcal{U}(\gamma,\pi)\Big|\;d\pi_{k}(\gamma)+\int_{\Omega\setminus
\Omega_{M, N}}\Big|\mathcal{U}(\gamma,\pi_{k})-\mathcal{U}(\gamma,\pi)\Big|\;d\pi_{k}(\gamma)\\
&:=I_{0,k}+I_{1,k}.
\end{split}
\end{equation}
 Let us look at  $I_{0,k}$.  We know that $\{\mathcal{U}(\gamma,
\pi_{k})\}$ is equicontinuous. 
 From the last estimate in (\ref{ineq: uniformity}), it is clear that $\{\mathcal{U}(\gamma,\pi_{k})\}_{k}$ is equibounded and equicontinuous. Therefore, by
Arzela-Ascoli, there exists a subsequence  $\{\mathcal{U}(\gamma,
\pi_{k_j})\}_{j}$ of $\{\mathcal{U}(\gamma, \pi_{k})\}_{k}$ that converges
uniformly. Now we claim that
\[
\mathcal{U}(\gamma, \pi_{k_j}) \longrightarrow \mathcal{U}(\gamma, \pi)
\quad \text{uniformly as}\quad j \to \infty.
\]

Ineed, since we saw the pointwise convergence of
$\{\mathcal{U}(\gamma,\pi_{k})\}$, then this sequence is Cauchy.  Suppose
now towards sake of a contradiction that $\{\mathcal{U}(\gamma,
\pi_{k_j})\}_{j}$ does not converge uniformly to $\mathcal{U}(\gamma,
\pi)$. Then we can find $\varepsilon>0$, such that for each integer $N>0$,
there exists some $j_{0}\geq N$ such that
\[
|\mathcal{U}(\gamma,\pi_{k_{j_{0}}})-\mathcal{U}(\gamma, \pi)|\geq
\varepsilon.
\]

As $\{\mathcal{U}(\gamma, \pi_{k})\}$ is Cauchy, then given any
$\varepsilon>0$, there exists a $K>0$  such that
\[
|\mathcal{U}(\gamma, \pi_{m})-\mathcal{U}(\gamma, \pi_{k})|<\varepsilon/2
\quad \text{for all }\quad m,k \geq K.
\]
The pointwise convergence of $\{\mathcal{U}(\gamma, \pi_{k})\}$ gives that
for any $\varepsilon>0$, we can find an integer $K^{\prime}>0$ such that
\[
|\mathcal{U}(\gamma, \pi_{k})-\mathcal{U}(\gamma, \pi)|<\varepsilon/2\quad
\text{for all }\quad k\geq K^{\prime}.
\]

Then taking $K^{\prime\prime}:=\max(K,K^{\prime})$,
\[
|\mathcal{U}(\gamma, \pi_{k_j})-\mathcal{U}(\gamma,\pi)|\leq
|\mathcal{U}(\gamma,\pi_{k_j})-\mathcal{U}(\gamma,
\pi_{m})|+|\mathcal{U}(\gamma,
\pi_{m})-\mathcal{U}(\gamma,\pi)|<\varepsilon \quad \text{for all} \;j,
m\geq K^{\prime\prime},
\]
the desired contradiction.  Therefore $\mathcal{U}(\gamma, \pi_{k_j})$
converges uniformly to the continuous $\mathcal{U}(\gamma, \pi)$. 
Relabeling $\{\mathcal{U}(\gamma, \pi_{k_j})\}$ to $\{\mathcal{U}(\gamma,
\pi_{k})\}$, subsequently $\{\mathcal{U}(\gamma, \pi_{k})\}$ converges
uniformly to $\mathcal{U}(\gamma, \pi)$ for all sufficiently large $k.$

This means that for any given $\varepsilon>0$, we can find an integer
$N>0$, not depending on $\gamma$, with $N>3/\varepsilon$ such that
\[
\sup_{\gamma \in X_{n}}|\mathcal{U}(\gamma, \pi_{k})-\mathcal{U}(\gamma,
\pi)|\leq \frac{1}{k}\leq \frac{1}{N}<\frac{\varepsilon}{3} \quad
\text{for all} \quad k\geq N \quad \text{sufficiently large}.
\]

So that then for each fixed $n$,
\begin{align*} I_{0,k}\leq \sup_{\gamma\in X_{n}}|\mathcal{U}(\gamma
,\pi_{k})-\mathcal{U}(\gamma, \pi)|&\leq \frac{\varepsilon}{3}\quad
\text{for all}\quad k \quad \text{sufficiently large.}
\end{align*}
Subsequently, the uniform convergence---hence strong convergence---allows us to conclude $I_{0,k}\to 0$ as $k \to \infty$.

Let us now turn to $I_{1,k}$. From each compact $\Omega_{M, N}$ arising from the
Condition \ref{condition:weaker} with $C/N>0$, let $\varepsilon>0$ be given
such that $C/N=:\varepsilon/3$. Then quite simply for each fixed
$N>0$,
\begin{align*}
I_{1,k}\leq \pi(\Omega\setminus \Omega_{M, N})\sup_{\gamma \in \Omega\setminus
\Omega_{M, N}}|\mathcal{U}(\gamma, \pi_{k})-\mathcal{U}(\gamma, \pi)|&\leq
\frac{C}{N}\Big(\sup_{\gamma \in \Omega\setminus
\Omega_{M, N}}|\mathcal{U}(\gamma, \pi_{k})|+\sup_{\gamma \in \Omega\setminus
\Omega_{M, N}}|\mathcal{U}(\gamma, \pi)|\Big)\\
&\leq 2 \frac{\varepsilon}{3} \quad \text{for all }\quad k.
\end{align*}
 Putting all this together in (\ref{eq: integralI0 integralI1}) and hence in (\ref{triangle ineq}) yields

 \[
 I_{k}\leq \frac{\varepsilon}{3}+2\frac{\varepsilon}{3}=\varepsilon, \quad
\text{for all }\quad k \quad \text{sufficiently large}.
 \]
The lemma is now proved.
\end{proof}

\begin{remark}
 Note that the above claim held true for general $\mathcal{K}$ having a $b$ modulus of continuity. But for the more specific example, if we define $\mathcal{K}$ as 
 \begin{align*}
     \mathcal{K}(\gamma, \sigma):=\theta \int_{0}^{1}e^{-\beta|\gamma(t)-\sigma(t)|^{2}}\;dt,
 \end{align*}
 where $\beta>0$ and $\theta>0$, the claim also holds true. In fact, since $e^{-|x|^{2}}$ is Lipschitz, 
\begin{align*}
\Big|e^{-|x|^{2}}-e^{-|y|^{2}}\Big|&\leq C|x-y|.
\end{align*}
Interchanging the values of $x$ and $y$ we achieve the Lipschitz estimate. Thus, $\exp\{-\beta|\sigma(t)-\gamma(t)|^{2}\}$ is Lipschitz.  Then  $\mathcal{K}(\gamma, \sigma):=\int_{0}^{1}e^{-\beta|\gamma(t)-\sigma(t)|^{2}}dt$ is uniformly Lipschitz in each coordinate, $|\mathcal{K}(\gamma_{1},\sigma)-\mathcal{K}(\gamma_{2},\sigma)|\leq C\|\gamma_{1}-\gamma_{2}\|$. Then 
\begin{align*}
    |\mathcal{U}(\gamma_{1}, \pi)-\mathcal{U}( \gamma_{2}, \pi)|&=\Big|\int_{\Omega}\mathcal{K}(\gamma_{1}, \sigma)d\pi(\sigma)-\int_{\Omega}\mathcal{K}(\gamma_{2}, \sigma)d\pi(\sigma)\Big|\\
&\leq \int_{\Omega}\Big|\mathcal{K}(\gamma_{1}, \sigma)-\mathcal{K}(\gamma_{2}, \sigma)\Big|d\pi(\sigma)\\
&\leq C\|\gamma_{1}-\gamma_{2}\|.
\end{align*}
\end{remark}

\subsection{The dual problem}\label{The dual problem} For the rest of the paper we only consider the interactions given by $\mathcal{K}(\gamma, \sigma)=\theta\int_{0}^{1}\exp\{-\beta|\gamma(t)-\sigma(t)|^{2}\}\;dt$, for some $\theta, \beta>0$. In order for the results here to be as ``smooth" as possible, we will make a small notational change to the functional $\mathcal{E}(\pi)$, (\ref{interaction term}). Namely, without losing generality, 
 \begin{align}\label{new iteration term}
        \mathcal{E}(\pi):= \int_{\Omega}c(\gamma)d\pi(\gamma)+\int_{\Omega}2\;\mathcal{U}(\gamma, \pi)d\pi(\gamma).
     \end{align}
The difference is the factor of $2$ in front of $\mathcal{U}(\gamma, \pi)$. 

We start with a heuristic discussion of the dual problem and Lagrange multipliers (we make this more rigorous in the next section). Just like in Section \ref{The dual problem and the endpoint cost function} we wish to minimize a functional subject to linear constraints. The novelty here is we  have a nonlinear functional. Let us elaborate on this, the interaction term $\int_{\Omega}\Big(2\;\int_{\Omega}\mathcal{K}(\gamma, \sigma)\;d\pi(\sigma)\Big)\;d\pi(\gamma)$ is actually  \textit{quadratic} with respect to $\pi$. We venture into what we did in the beginning of Section \ref{The dual problem and the endpoint cost function} to produce the required constraint of the optimal path  Kantorovich duality with interaction. For the moment let us abandon rigor and see where this takes us---it will take us to the correct dual problem when we look at interaction terms.

For $(\phi, \psi)\in C_{c}^{0}(X)\times C_{c}^{0}(X)$ and $\pi\in \mathcal{M}^{+}(\Omega)$, define
\begin{align*}
\mathcal{L}(\pi; \phi, \psi ):=\int_{\Omega}c(\gamma)d\pi(\gamma)+\int_{\Omega}2\;\mathcal{U}(\gamma, \pi)d\pi(\gamma)&-\int_{\Omega}\phi(\gamma(0))d\pi(\gamma)+\int_{X}\phi(x)d\mu_{0}(x)\\
&-\int_{\Omega}\psi(\gamma(1))d\pi(\gamma)+\int_{X}\psi(y)d\mu_{1}(y).
\end{align*}
After rearranging 
\begin{align*}
    \mathcal{L}(\pi; \phi, \psi)&=\int_{\Omega}\textbf{\Big(}c(\gamma)+2\;\mathcal{U}(\gamma, \pi)-(\phi(\gamma(0))+\psi(\gamma(1)))\textbf{\Big)}\;d\pi(\gamma)\\
    &\quad\quad\quad\quad\quad+\int_{X}\phi(x)d\mu_{0}(x)+\int_{X}\psi(y)d\mu_{1}(y)\\
    &=\mathcal{Q}(\pi;\phi,\psi)+\int_{X}\phi(x)d\mu_{0}(x)+\int_{X}\psi(y)d\mu_{1}(y),
\end{align*}
where
\begin{align*}
    \mathcal{Q}(\pi;\phi,\psi):=\int_{\Omega}c(\gamma)+2\;\mathcal{U}(\gamma, \pi)-\left(\phi(\gamma(0))+\psi(\gamma(1))\right)\;d\pi(\gamma).
\end{align*}
Notice that $\mathcal{Q}(\pi;\phi,\psi)$ contains the quadratic term that we talked about in the prequel. Then the dual function is
\begin{align}\label{eq: new dual}
        \mathcal{D}(\phi, \psi):=\inf_{\pi\geq 0}\left\{\mathcal{Q}(\pi;\phi,\psi)+\int_{X}\phi(x)\;d\mu_{0}(x)+\int_{X}\psi(y)\;d\mu_{1}(y)\right\}.
    \end{align}
The business at hand is to minimize the quadratic term $\mathcal{Q}(\pi;\phi,\psi)$ with respect to $\pi$. But how can one minimize such quadratic term in $\pi$? This can be answered if we recognize this as minimizing a quadratic functional in infinite dimensions and compare to the more tangible problem in finite dimension.

The minimization of $\mathcal{Q}(\pi;\phi,\psi)$ has the form: $\min_{p\geq 0}q(p)$ for 
\begin{align}\label{eq: inner product}
        q(p):=\langle b, p\rangle+\langle Ap, p\rangle,
    \end{align}
where $\langle\; ,\; \rangle$ is the Euclidean inner product, $b\in \mathbb{R}^{n}$ is positive, and $A\in \mathbb{R}^{n \times n}$ a positive semi-definite matrix. The term $\langle Ap, p\rangle$ can be interpreted as the quadratic term in (\ref{eq: new dual}), that is, the term $\int_{\Omega}\mathcal{U}(\gamma, \pi)d\pi(\gamma)$.  Futher, $p$ represents $\pi$, while $Ap$ represents $2\;\mathcal{U}(\gamma, \pi)$, and $b$ represents $c(\gamma)-\phi(\gamma(0))-\psi(\gamma(1))$. Then 
\[
\langle b, p\rangle:=\Big\langle c(\gamma)-(\phi(\gamma(0))+\psi(\gamma(1))), \pi \Big\rangle=\int_{\Omega}\textbf{(}c(\gamma)-(\;\phi(\gamma(0))+\psi(\gamma(1))\;)\textbf{)}\;d\pi(\gamma);
\]
\[
\langle Ap, p\rangle:=\Big\langle 2\;\mathcal{U}(\gamma, \pi),\pi \Big\rangle=\int_{\Omega}2\textbf{\Big(}\mathcal{U}(\gamma, \pi)\textbf{\Big)}\;d\pi(\gamma).
\]
Thanks to Fubini and Tonelli the inner products above make (kind of) sense and their sum equal $\mathcal{Q}(\pi;\phi,\psi)$. Then to minimize $\mathcal{Q}(\pi;\phi,\psi)$ with respect to $\pi \geq 0$ in $\mathcal{P}(\Omega)$, it suffices to minimize (\ref{eq: inner product}) with respect to $p\geq 0$.  Take a derivative of $q(p)$ with respect to $p$, set it equal to zero, and solve for minimum. To be cautious we will use vector notation. Take a gradient with respect to $p$
\[
\nabla_{p} q(p)=b+2Ap.
\]
We have the constraint $p\geq 0$, so at the minimizer $p_{\text{min}}$ we have, by the Karush-Kuhn-Tucker (KKT) conditions \cite[Ch 10.2.2]{solomon2015numerical},
\begin{align*}
\nabla_{p} q(p) & \geq 0 \textnormal{ (coordinate wise)},\\
\nabla_{p_i}q(p) & = 0 \textnormal{ if } p_i>0.
\end{align*}
 
Let us interpret this in our infinite dimensional problem, here $\pi_{\text{min}}$ corresponds to the $p_{\text{min}}$, and we have
\begin{align*}
\mathcal{U}(\gamma, \pi_{\text{min}} )  & \geq -\frac{1}{2}(c(\gamma)-(\phi(\gamma(0))+\psi(\gamma(1))), \;\forall\;\gamma \in \Omega\\
\mathcal{U}(\gamma, \pi_{\text{min}} )  & =   -\frac{1}{2}(c(\gamma)-(\phi(\gamma(0))+\psi(\gamma(1))), \;\forall\;\gamma \in \text{spt}(\pi_{\textnormal{min}}).
\end{align*}
Rearranging, this can be written as 
\begin{align}\label{ine: new dual inequality}
\phi(\gamma(0))+\psi(\gamma(1)) & \leq c(\gamma)+2\;\mathcal{U}(\gamma, \pi_{\text{min}}),\;\forall\;\gamma \in \Omega,\\
\phi(\gamma(0))+\psi(\gamma(1)) & =  c(\gamma)+2\;\mathcal{U}(\gamma, \pi_{\text{min}}),\;\forall\;\gamma \in \text{spt}(\pi_{\text{min}}).
\end{align}
Thus, we have a good guess about the condition that characterizes the optimal $\pi$ in \eqref{eq: new dual}. More concretely, and mimicking what we have already done in Section \ref{The dual problem and the endpoint cost function}, if there exists a path $\gamma$ for which $c(\gamma)+2\;\mathcal{U}(\gamma, \pi)-(\phi(\gamma(0))+\psi(\gamma(1)))<0$, then we can make $\min_{\pi\geq 0}\mathcal{Q}(\pi;\phi,\psi)=-\infty$ and in turn make $\mathcal{D}(\phi, \psi)=-\infty$. Otherwise, for all $\sigma\in \Omega$, and keeping $\gamma$ fixed (remembering the proof of Proposition \ref{prop: dual problem}), $\phi(\gamma(0))+\psi(\gamma(1))\leq c(\gamma)+2\;\mathcal{U}(\gamma, \pi)$ $\pi$-a.e. 

Note that for any $\pi\in \mathcal{P}(\Omega)$, we may consider the auxiliary dual problem, 
\begin{align}\label{new dual problem}
  D_{\pi}(\phi, \psi)=\sup \left\{\int_{X}\phi(x)\;d\mu_{0}(x)+\int_{X}\psi(y)\;d\mu_{1}(y) :\phi(\gamma(0))+\psi(\gamma(1))\leq c(\gamma)+2\;\mathcal{U}(\gamma; \pi)\right\} 
\end{align}

For the moment let us assume that Theorem \ref{theorem:duality}'s equality holds for (\ref{new iteration term}) having constraint (\ref{ine: new dual inequality}). Suppose the following holds:
\begin{align}\label{eqn: new duality}
        \inf\left\{\mathcal{E}(\pi):\;\pi \in \Pi_{\text{path}}(\mu_{0},\mu_{1})\right\}=\sup\left\{\mathcal{D}(\phi,\psi)\right\}
    \end{align}
Suppose $(\phi_{0}, \psi_{0})$ maximizes (\ref{new dual problem}) and $\pi_{0}$ is optimal for (\ref{new iteration term}). Then, thanks to (\ref{ine: new dual inequality}-30)
$c(\gamma)+2\;\mathcal{U}(\gamma, \pi_{0})=\phi_{0}(\gamma(0))+\psi_{0}(\gamma(1))$. Furthermore, strong duality (\ref{eqn: new duality}) ensures 
\begin{align*}
        \int_{\Omega}\left[c(\gamma)+2\;\mathcal{U}(\gamma, \pi_{0})\right]d\pi_{0}(\gamma)&=\int_{X}\phi_{0}(x)d\mu_{0}(x)+\int_{X}\psi_{0}(y)\;d\mu_{1}(y)\\
        &=\int_{\Omega}\textbf{[}\phi_{0}(\gamma(0))+\psi_{0}(\gamma(1))\textbf{]}d\pi_{0}(\gamma).
    \end{align*}

Equivalently 
\begin{align*}
        \int_{\Omega} \textbf{[}c(\gamma)+2\;\mathcal{U}(\gamma, \pi_{0})-(\phi_{0}(\gamma(0))+\psi_{0}(\gamma(1)))\textbf{]}d\pi_{0}(\gamma)=0.
    \end{align*}
The integrand is nonnegative by (\ref{ine: new dual inequality}), so it has to vanish $\pi_{0}$-a.e.
Moreover, if we further stipulate something very similar like the endpoint cost function (\ref{effective infimum}) to $c(\gamma)+2\;\mathcal{U}(\gamma, \pi_{0})$, say (\ref{effectivecost}) for which the constraint $\phi_{0}(x)+\psi_{0}(y)\leq \inf_{\gamma(0)=x, \gamma(1)=y}\;c(\gamma)+2\;\mathcal{U}(\gamma, \pi_{0})$ holds $\pi_{0}$-a.e., then we would have
\begin{align*}
        c(\gamma)+2\;\mathcal{U}(\gamma, \pi_{0})&=\phi_{0}(\gamma(0))+\psi_{0}(\gamma(1))\\
       &=\phi_{0}(x)+\psi_{0}(y)\\
       &\leq \inf_{\gamma(0)=x,\; \gamma(1)=y}\textbf{[}c(\gamma)+2\;\mathcal{U}(\gamma, \pi_{0})\textbf{]},
    \end{align*}
which in turn would imply $c(\gamma)+2\;\mathcal{U}(\gamma, \pi_{0})=\inf_{\gamma(0)=x,\; \gamma(1)=x}\textbf{[}c(\gamma)+2\;\mathcal{U}(\gamma, \pi_{0})\textbf{]}$. 

With this in mind, we state the interaction analogue of Theorem \ref{theorem:duality}, i.e., we establish (\ref{eqn: new duality}), but prove it in the end of Section \ref{supergradient of concave function} after some preliminary discussion and a bit of mathematical machinery coming from Definition \ref{def: supergradient}.

\begin{theorem}\label{theorem: interaction analogue}
Strong duality holds. Namely, 
\begin{align*}
    \inf_{\pi\in \Pi_{\text{path}}(\mu_{0},\mu_{1})}\int_{\Omega}\Big(c(\gamma)+2\;\mathcal{U}(\gamma; \pi)\Big)d\pi(\gamma)=\sup_{\phi, \psi}\mathcal{D}(\phi,\psi).
\end{align*}
 \end{theorem}

\subsection{Supergradient of a concave function} \label{supergradient of concave function} We give a rigorous account to our discussion in Section \ref{The dual problem}. The background material needed in this section comes from Rockafellar \cite[\S 23]{rockafellar2015convex}. Namely, we use the fact that $x$ is in the maximization set of $f$---a concave function---if and only if $0$ is in the supergradient of $f$. 

From (\ref{eq: new dual}), let $q(\phi, \psi):=\inf_{\pi\geq 0}\mathcal{Q}(\pi;\phi,\psi)$. So that  for all $\phi, \psi$ there is a certain $\pi_{*}\geq 0$ such that, 
\begin{align}\label{Quadratic term in pi}
q(\phi, \psi)=\mathcal{Q}(\pi_{*};\phi,\psi)\leq \mathcal{Q}(\pi;\phi,\psi),\quad \forall \pi\geq 0.
\end{align}
 In this case, from equation (\ref{eq: new dual}),
 \[
 \mathcal{D}(\phi, \psi):=\inf_{\pi\geq 0}\mathcal{L}(\pi; \phi,\psi)=\inf_{\pi\geq 0}\mathcal{Q}(\pi;\phi,\psi)+\mathcal{J}(\phi, \psi)=q(\phi, \psi)+\mathcal{J}(\phi, \psi),
\]
where $\mathcal{J}(\phi, \psi)=\int\phi\;d\mu_{0}+\int\psi \;d\mu_{1}$.
It is immediate to see that $\mathcal{J}(\phi, \psi)$ is linear in $\phi, \psi$. For each fixed $\pi\geq 0$, $\mathcal{Q}(\pi;\phi,\psi)$ is linear in $\phi, \psi$ and so $q(\phi, \psi)$ is a minimum of a family of linear functionals. In other words, for all $(\phi_{0}, \psi_{0})$ we can find $\pi_{0}$ such that $q(\phi_{0}, \psi_{0})=\mathcal{Q}(\pi_{0};\phi_{0},\psi_{0})$ and $q(\phi, \psi)\leq \mathcal{Q}(\pi;\phi,\psi)$. Thus, $q$ is concave in $\phi, \psi$. The superdifferentiability analogue of $\mathcal{D}(\phi, \psi)$ is given---just as in \cite[\S23]{rockafellar2015convex}.

\begin{remark}\label{def: supergradient}
The \emph{superdifferential} of $\mathcal{D}$ is made of functions of the form 
\begin{align*}\mathfrak{l}_{\pi}(\phi, \psi):=-\int_{\Omega}[\phi(\gamma(0))+\psi(\gamma(1))]\;d\pi(\gamma)+\int_{X}\phi(x)d\mu_{0}(x)+\int_{X}\psi(y)d\mu_{1}(y),
\end{align*}
for some $\pi\geq 0$. In particular, $\mathfrak{l}_{\pi}$ is in $\partial \mathcal{D}(\phi_{0}, \psi_{0})$
if
\[
 \mathcal{D}(\phi, \psi)\leq \mathcal{D}(\phi_{0}, \psi_{0})+\mathfrak{l}_{\pi}(\phi-\phi_{0}, \psi-\psi_{0})\quad\forall (\phi, \psi).
\]
Notice the pair $(\phi_{0}, \psi_{0})$ is a maximum of $\mathcal{D}$ if and only if the $0$ functional is in $\partial\mathcal{D}(\phi_{0},\psi_{0})$, i.e., if there is a $\pi$ such that $\mathfrak{l}_{\pi}(\phi,\psi)=0$ for all $(\phi,\psi).$
\end{remark}
 This aligns perfectly well with Rockafellar's statement \cite{rockafellar2015convex}: $\pi$---having the form $\mathfrak{l}_{\pi}$---is in the maximization set of $\mathcal{D}(\phi, \psi)$ for all $(\phi, \psi)$ if and only if  $0$ is in the supergradient of $\mathcal{D}$. That is, if and only if $\mathfrak{l}_{\pi}$ is identically zero.
 
A consequence of this indicates that if there exists $\pi\geq 0$ such that for all $\phi, \psi$
\begin{align*}
    0=\mathfrak{l}_{\pi}=-\int_{\Omega}[\phi(\gamma(0))+\psi(\gamma(1))]d\pi(\gamma)+\int_{X}\phi(x)d\mu_{0}(x)+\int_{X}\psi(y)d\mu_{1}(y),
\end{align*}
then, $(e_{0})_{\sharp}\pi=\mu_{0}$ and $(e_{1})_{\sharp}\pi=\mu_{1}$ and thusly $\pi$ is admissible.

Armed with this discussion, we are ready to prove Theorem \ref{theorem: interaction analogue}. 

\begin{proof}[\textbf{Proof of Theorem \ref{theorem: interaction analogue}}]
 Let $(\phi_{*}, \psi_{*})$ be the maximum of $\mathcal{D}(\phi, \psi)$. In view of the above discussion (and Remark \ref{def: supergradient}), there exists $\pi_{*}\geq 0$ such that $\mathfrak{l}_{\pi_{*}}\in \partial \mathcal{D}(\phi_{*}, \psi_{*})$ and $\mathfrak{l}_{\pi_{*}}=0$. Equivalently, $q(\phi_{*}, \psi_{*})=\mathcal{Q}(\pi_{*};\phi_{*},\psi_{*})$. In particular, $\pi_{*}\in\Pi_{\text{path}}(\mu_{0},\mu_{1})$. Then, for all $(\phi, \psi)$ and all admissible $\pi$, and applying  the inequality (\ref{Quadratic term in pi}) we have
 
 \begin{align*}
 \int_{\Omega}c(\gamma)+2\;\mathcal{U}(\gamma,\pi_{*})\;d\pi_{*}(\gamma)&=\mathcal{L}(\pi_{*},\phi_{*},\psi_{*})\\
 &=\mathcal{Q}(\pi_{*};\phi_{*},\psi_{*})+\int_{X}\phi_{*}(x)d\mu_{0}(x)+\int_{X}\psi_{*}(y)d\mu_{1}(y)\\
 &\leq \mathcal{Q}(\pi;\phi_{*},\psi_{*})+\int_{X}\phi_{*}(x)d\mu_{0}(x)+\int_{X}\psi_{*}(y)d\mu_{1}(y)\\
 &=\int_{\Omega}c(\gamma)+2\;\mathcal{U}(\gamma,\pi)\;d\pi(\gamma)\quad \forall\;\pi\quad\text{admissible}
 \end{align*}
 The latter inequality indicates $\pi_{*}$ minimizes  $\int (c(\gamma)+2\;\mathcal{U}(\gamma; \pi))d\pi$. So since $\pi_{*}$ is optimal, then
\begin{align*}
     \phi_{*}(\gamma(0))+\psi_{*}(\gamma(1))&=c(\gamma)+2\;\mathcal{U}(\gamma; \pi_{*}),\quad \text{in the support of}\; \pi_{*}.
\end{align*}
Therefore, 
\begin{align*}
\int_{X}\phi_{*}(x)\;d\mu_{0}(x)+\int_{X}\psi_{*}(y)\;d\mu_{1}(y)&=\int_{\Omega}(\phi_{*}(\gamma(0))+\psi_{*}(\gamma(1)))\;d\pi_{*}(\gamma)\\
&=\int_{\Omega}(c(\gamma)+2\;\mathcal{U}(\gamma; \pi_{*}))\;d\pi_{*}(\gamma)
\end{align*}
$\pi_{*}$-a.e. But $(\phi_{*},\psi_{*})$ solves the maximization (\ref{eq: new dual}); thus
\begin{align*}
    \sup\mathcal{D}(\phi, \psi)=\inf_{\pi\in \Pi_{\text{path}}}\mathcal{E}(\pi).
\end{align*}
\end{proof}
 \begin{remark}
 The above result says that any solution $(\phi_{*}. \psi_{*})$ to the dual problem  produces a solution  $\pi_{*}$ to the primal problem.
 \end{remark}
 A consequence of strong duality is a characterization of optimizers using the auxiliary dual problem. This is the content of the next result.
 
 \begin{theorem}
 If $\pi_{0}$ solves the primal problem, then $(\phi_{0}, \psi_{0})$ solves the auxiliary dual problem. 
 \end{theorem}
 
 \begin{proof}
 First of all using the auxiliary dual problem (\ref{new dual problem}), $D_{\pi}(\phi, \psi)\leq \mathcal{E}(\pi)$ for all $(\phi, \psi)$; for all admissible $\pi$. For the reverse inequality, we proceed as follows. Let $\pi_{0}$ be some minimizer for the primal problem. Define $(\phi_{0}, \psi_{0})$ maximizing (\ref{new dual problem}). From the previous Section \ref{The dual problem}, any pair $(\phi_{0}, \psi_{0})$ that achieves the supremum are such that
 \begin{align}\label{Rockafellar condition}
 \phi_{0}(\gamma(0))+\psi_{0}(\gamma(1))=c(\gamma)+2\;\mathcal{U}(\gamma; \pi_{0})\quad \forall \gamma\in \text{spt}(\pi_{0}).
 \end{align}
 This implies, as $\pi_{0}\in \Pi_{\text{path}}(\mu_{0}, \mu_{1})$,
 \begin{align*}
     D_{\pi}(\phi_{0}, \psi_{0})&=\int_{X}\phi_{0}(x)\;d\mu_{0}(x)+\int_{X}\psi_{0}(y)\;d\mu_{1}(y)\\
     &=\int_{\Omega}(\phi_{0}(\gamma(0))+\psi_{0}(\gamma(1)))\;d\pi_{0}(\gamma)\\
     &=\int_{\Omega}(c(\gamma)+2\;\mathcal{U}(\gamma; \pi_{0}))\;d\pi_{0}(\gamma)\\
     &=\mathcal{E}(\pi_{0}).
 \end{align*}
 The second to last equation follows from (\ref{Rockafellar condition}).
 
 Remembering the Lagrangian multipliers functional from Section \ref{The dual problem},
 \begin{align*}
     \mathcal{L}(\pi_{0};\phi_{0},\psi_{0})&=\int_{\Omega}(c(\gamma)+2\;\mathcal{U}(\gamma; \pi_{0})-(\phi_{0}(\gamma(0))+\psi_{0}(\gamma(1)))\;d\pi_{0}(\gamma)+\int_{X}\phi_{0}\;d\mu_{0}+\int_{X}\psi_{0}\;d\mu_{1},
 \end{align*}
 and thanks to (\ref{Rockafellar condition}) the first integral vanishes (and so $\mathcal{Q}(\pi_{0};\phi_{0},\psi_{0})$ vanishes as well), and thus
 \begin{align*}
     \mathcal{L}(\pi_{0};\phi_{0},\psi_{0})&=\int_{X}\phi_{0}(x)\;d\mu_{0}(x)+\int_{X}\psi_{0}(y)\;d\mu_{1}(y)\\
     &=D_{\pi}(\phi_{0}, \psi_{0}).
 \end{align*}
 Hence, from (\ref{eq: new dual}), we have
 \begin{align*}
     \mathcal{D}(\phi_{0}, \psi_{0})\leq\mathcal{L}(\pi_{0};\phi_{0},\psi_{0})=\mathcal{E}(\pi_{0}).
 \end{align*}
 But by the KKT conditions, the inequality becomes an equality, and hence $\mathcal{D}(\phi_{0}, \psi_{0})=\mathcal{E}(\pi_{0})$.
 \end{proof}

\subsection{The effective cost}\label{The effective cost} Toward the end of Section \ref{section:minimal paths} we introduced the \emph{endpoint function} (\ref{effective infimum}) which represented the minimum cost function of a path going from $x$ to $y$. Now for the problem with interaction we will introduce what we call an ``effective" endpoint cost function that includes the interaction term. 

Therefore given $\pi$, define $c_{0}(\gamma):=c(\gamma)+2\;\mathcal{U}(\gamma; \pi)$ as the \textit{effective cost}. The analogue of the endpoint cost function (\ref{effective infimum}) is
\begin{align}\label{effectivecost}
    c_{\text{eff}}(x, y):=\inf_{\gamma(0)=x, \gamma(1)=y}\;c_{0}(\gamma).
\end{align}
In what follows we deal explicitly with 
\begin{align*}
    c(\gamma)=\int_{0}^{1}\frac{1}{2}|\dot{\gamma}(t)|^{2}dt,\quad \text{and}\quad \mathcal{U}(\gamma;\pi):=\theta \int_{\Omega}\int_{0}^{1}e^{-|\gamma(t)-\sigma(t)|^{2}}dt\;d\pi(\sigma),
\end{align*}
for some $\theta<<1$ to be determined later. We start by showing that the optimal path of $c_{0}(\gamma)$ is sufficiently close to the  optimal path of $c(\gamma)$, if $\theta$ is small. 

\begin{proposition}\label{prop on effective cost}
Fix $\pi\in \mathcal{P}(\Omega)$ and $\theta>0$. Then for any path $\gamma: [0, 1]\to X$ which is Lipschitz in $[0, 1]$ and twice differentiable in $(0, 1)$ define, 
\begin{align*}
    c_{0}(\gamma)=\int_{0}^{1}\frac{1}{2}|\dot{\gamma}(t)|^{2}dt+2\theta \int_{0}^{1}\int_{\Omega}e^{-|\gamma(t)-\sigma(t)|^{2}}d\pi(\sigma)\;dt
\end{align*}
Given $x, y$, let $\gamma_{0}$ be optimal for $c_{0}(\gamma)$. Then $\gamma_{0}$ is of the form $\gamma_{0}(t)=\gamma_{euc}(t)+\theta E(\theta, t)$; where $\gamma_{euc}$ is optimal for $c(\gamma)$ and where $E(\theta, t)$ satisfies the estimate
\begin{align*}
    \sup_{0\leq t\leq 1}|E(\theta, t)|\leq \sqrt{n}.
\end{align*}
\end{proposition}
\begin{proof} Notice that $d\pi_{t}(z):=(e_{t})_{\sharp}\pi,$ therefore
\begin{align*}
\int_{\Omega}e^{-|\gamma(t)-\sigma(t)|^{2}}d\pi(\sigma)=\int_{\mathbb{R}^{n}}e^{-|\gamma(t)-z|^{2}}d\pi_{t}(z),\;\text{for all}\; \gamma(t)\;\text{and all }\; t\in [0, 1];
\end{align*}
so for every $\pi$, the interaction term produces a potential function,
\begin{equation*}
    V(x,t)=-2\;\theta\int_{\mathbb{R}^{n}}e^{-\left|x-z\right|^{2}}\;d\pi_{t}(z).
\end{equation*}
We quickly review the Euler-Lagrange equation  solutions of critical points of $c_{0}(\gamma)$. Fix $\varepsilon>0$ and consider $\varphi: [0, 1]\to X$, a Lipschitz, twice differentiable function in $(0, 1)$ such that $\varphi(0)=\varphi(1)=0$. Perturbing the minimizer $\gamma_{0}$ by $\varphi \in C^{2}([0, 1])$, we get $\gamma_{\varepsilon}:=\gamma_{0}+\varepsilon\varphi$. Then looking at the Euler-Lagrange equation we have
\begin{align*}\label{eqn: Euler-Lagrange}
        0=\frac{d}{d\varepsilon}\Big|_{\varepsilon=0}c_{0}(\gamma_{\varepsilon})=\int_{0}^{1}\Big(-\ddot{\gamma_{0}}(t)-4\theta\int_{\mathbb{R}^{n}}e^{-|\gamma_{0}(t)-z|^{2}}(\gamma_{0}(t)-z)\;d\pi_{t}(z)\Big)\cdot \varphi(t)\;dt
    \end{align*}
for all $\varphi\in C^{2}([0, 1])$ for all $0\leq t \leq 1$. Hence, $\gamma_{0}$ solves the equation
\[
\ddot{\gamma_{0}}(t)=-4\theta\int_{\mathbb{R}^{n}}e^{-|\gamma_{0}(t)-z|^{2}}(\gamma_{0}(t)-z)d\pi_{t}(z).
\]
We estimate the exponential expression inside the integral.  For radial $r=|x-z|$, \[
\Big|e^{-|x-z|^{2}}(x-z)\Big|\leq e^{-r^{2}}r\leq 1.\] Then 
\begin{align*}
    |\ddot{\gamma_{0}}(t)|\leq 4\theta\int_{\mathbb{R}^{n}}\Big|e^{-|\gamma_{0}(t)-z|^{2}}(\gamma_{0}(t)-z)\Big|d\pi_{t}(z)\leq 4\theta.
\end{align*}
So thanks to Proposition \ref{error prop}, we have that $\gamma_{0}$ is of the form $\gamma_{0}(t)=\gamma_{euc}(t)+\theta\;E(\theta, t)$.
\end{proof}

For our last result, we need the following derivative calculation:

\begin{proposition}\label{prop on gradient effective cost}
Let $\gamma_{0}$ be a minimizer of $c_{0}(\gamma)$ in $\Omega$. For all $0\leq t \leq 1$, let $\gamma_{h}(t):=\gamma_{0}(t)+th\hat{e}$ be a path from $\gamma_{h}(0)=\gamma_{0}(0)$ to $\gamma_{h}(1)=\gamma_{0}(1)+h\hat{e}$, where $h\neq 0$ and $\hat{e}:=\langle 0, \ldots, 1, \ldots, 0\rangle$ is a unit vector in $\mathbb{R}^{n}$. Then for all $y$
\begin{align}\label{eqn: gradient effective}
    \nabla_{y}c_{\text{eff}}(x,y)=x-y-4\theta\int_{0}^{1}\int_{\mathbb{R}^{n}}t\exp\{-|\gamma_{0}(t)-z|^{2}\}(\gamma_{0}(t)-z)\;d\pi_{t}(z)\;dt.
\end{align}

\end{proposition}

\begin{proof}
The proof of Lemma \ref{effective cost is differentiable} can be incorporated here with $c_{\textbf{e}}$ replaced by (\ref{effectivecost}).
\end{proof}

We now provide a Lipschitz bound on $T_{\mathbf{z}}(\mathbf{w}):=e^{-|\mathbf{w}-\mathbf{z}|^{2}}(\mathbf{w}-\mathbf{z})$ that will help us prove Lemma \ref{c-chypersurface}.

\begin{proposition}\label{prop on Lipschitz bound}
Given $\mathbf{w},\mathbf{u}, \mathbf{z} \in \mathbb{R}^{n}$, we have
\begin{align*}
    |T_{\mathbf{z}}(\mathbf{w})-T_{\mathbf{z}}(\mathbf{u})|\leq n^{2}|\mathbf{w}-\mathbf{u}|.
\end{align*}
\end{proposition}
\begin{proof} We firstly bound a partial derivative of $T_{\mathbf{z}}(\mathbf{w})$. We have 
\[
\partial_{w_{j}}T_{\mathbf{z}}^{i}(\mathbf{w})=e^{-|\mathbf{w}-\mathbf{z}|^{2}}(\delta_{ij}-2(w_{i}-z_{i})(w_{j}-z_{j})).
\]
For radial $r=|\mathbf{w}-\mathbf{z}|$, we get an upper bound, 
\[
\Big|\partial_{w_{j}}T_{\mathbf{z}}^{i}(\mathbf{w})\Big|\leq e^{-|\mathbf{w}-\mathbf{z}|^{2}}(1+|\mathbf{w}-\mathbf{z}|^{2}) \leq e^{-r^{2}}(1+r^{2})\leq 1.
\]

Secondly,  borrowing notation and calculations from Spivak's \emph{Calculus on Manifolds} \cite{spivak2018calculus}, we observe that
\[
T_{\mathbf{z}}^{i}(\mathbf{w})-T_{\mathbf{z}}^{i}(\mathbf{u})=\sum_{j=1}^{n}\left[T_{\mathbf{z}}^{i}(w_{1},\ldots,w_{j};u_{j+1},\ldots, u_{n})-T_{\mathbf{z}}^{i}(w_{1},\ldots,w_{j-1};u_{j},\ldots, u_{n})\right].
\]
As $T_{\mathbf{z}}(\mathbf{w})$ is continuously differentiable, the mean value theorem gives us for every $i$ and $j$
\begin{align*}
    T_{\mathbf{z}}^{i}(w_{1},\ldots, w_{j};u_{j+1},\ldots, u_{n})-T_{\mathbf{z}}^{i}(w_{1},\ldots, w_{j-1};u_{j},\ldots, u_{n})=(w_{j}-u_{j})\partial_{w_{j}}T_{\mathbf{z}}^{i}(v_{ij}),
\end{align*}
for some $v_{ij}$. The absolute value of the  right-hand side of this is  $|w_{j}-u_{j}|\;|\partial_{w_{j}}T_{\mathbf{z}}^{i}(v_{ij})|\leq |w_{j}-u_{j}|$. Then $|T_{\mathbf{z}}^{i}(\mathbf{w})-T_{\mathbf{z}}^{i}(\mathbf{u})|\leq n|w_{j}-u_{j}|\leq n|\mathbf{w}-\mathbf{u}|$, since each $|w_{j}-u_{j}|\leq|\mathbf{w}-\mathbf{u}|$. Hence, $|T_{\mathbf{z}}(\mathbf{w})-T_{\mathbf{z}}(\mathbf{u})|\leq \sum_{i=1}^{n}|T_{\mathbf{z}}^{i}(\mathbf{w})-T_{\mathbf{z}}^{i}(\mathbf{u})|\leq\sum_{i=1}^{n}n|\mathbf{w}-\mathbf{u}|=n^{2}|\mathbf{w}-\mathbf{u}|,$ and the proposition follows. 
\end{proof}

\begin{lemma}\label{c-chypersurface}
Let $\theta_{0}:=\frac{1}{\sqrt{2}}n^{-5/4}.$ If $\theta<\theta_{0}$, then for all $y\in \mathbb{R}^{n}$, and $x_{1}\neq x_{2}$,
\begin{align*}\label{neq}
    \nabla_{y}\Big(c_{\text{eff}}(x_{2},y)-c_{\text{eff}}(x_{1},y)\Big)\neq 0.
\end{align*}
\end{lemma}

\begin{remark}\label{pi independent of effective cost}
The above holds independent of $\pi$.
\end{remark}

\begin{proof}
 For all $\theta>0$, let $\gamma_{\theta, i}(t)=\gamma_{0,i}(t)+\theta E(\theta, t)$ be an optimal path from $x_{i}$ to $y$ with respect to $c_{0}(\gamma_{h})$; where $\gamma_{0,i}(t)=x_{i}+t(y-x_{i})$---geodesic path for $c(\gamma)$---for $i=1,2$ and $\theta\cdot E(\theta, t)$ an error term. From Proposition \ref{prop on gradient effective cost}, we get for all $y$
\begin{align*}
        \nabla_{y}\Big(c_{\text{eff}}(x_{2},y)-c_{\text{eff}}(x_{1},y)\Big)=x_{1}-x_{2}-4\theta\int_{0}^{1}\int_{\mathbb{R}^{n}}t\Big[T_{z}(\gamma_{\theta, 2}(t))-T_{z}(\gamma_{\theta, 1}(t))\Big]d\pi_{t}(z)dt.
    \end{align*}
Thanks to Proposition \ref{prop on Lipschitz bound}, we can estimate the expression inside the brackets. Namely, since we saw $|\ddot{\gamma_{0}}(t)|\leq 4 \theta$, Proposition \ref{error prop} applies to give 
\begin{align*}
        \left|T_{z}(\gamma_{\theta,2}(t))-T_{z}(\gamma_{\theta,1}(t))\right|&\leq n^{2}|\gamma_{\theta, 2}(t)-\gamma_{\theta, 1}(t)|\\
        &\leq n^{2}\max_{0\leq t \leq 1}|\gamma_{\theta, 2}(t)-\gamma_{\theta, 1}(t)|\\
        &\leq n^{\frac{5}{2}}\theta|x_{2}-x_{1}|.
    \end{align*}
In particular, 
\begin{align*}
    \left|4\theta\int_{0}^{1}\int_{\mathbb{R}^{n}}t[T_{z}(\gamma_{\theta,2}(t))-T_{z}(\gamma_{\theta, 1}(t))]d\pi_{t}(z)\;dt\right|&\leq 4\theta\int_{0}^{1}t\int_{\mathbb{R}^{n}}n^{\frac{5}{2}}\theta|x_{2}-x_{1}|d\pi_{t}(z)\;dt\\
    &=2n^{\frac{5}{2}}\theta^{2}|x_{2}-x_{1}|.
\end{align*}

Let $\theta_{0}$ be such that $\delta:=2 n^{\frac{5}{2}}\theta_{0}^{2}<1$. The reverse triangle inequality applies to show
\begin{align*}
        \left|\nabla_{y}\left(c_{\text{eff}}(x_{2},y)-c_{\text{eff}}(x_{1}, y)\right)\right|&\geq |x_{1}-x_{2}|-4\theta\left|\int_{0}^{1}t\int_{\mathbb{R}^{n}}\left[T_{z}(\gamma_{\theta, 2}(t))-T_{z}(\gamma_{\theta, 1}(t))\right]d\pi_{t}(z)\;dt\right|\\
        &\geq |x_{1}-x_{2}|-2n^{\frac{5}{2}}\theta^{2}|x_{2}-x_{1}|\\
        &=(1-2 n^{\frac{5}{2}}\theta^{2})|x_{1}-x_{2}|>(1-\delta)|x_{1}-x_{2}|.
    \end{align*}
Therefore, since $\delta<1$, $|\nabla_{y}\textbf{(}c_{\text{eff}}(x_{2},y)-c_{\text{eff}}(x_{1},y)\textbf{)}|>0$, and the proof is complete.

\end{proof}
Theorem \ref{result 4: uniqueness Kantorovich solution with interaction} now follows as a corollary of Lemma \ref{c-chypersurface} and Theorem \ref{result 2: uniqueness Kantorovich solution on paths}.

  \begin{proof}[\textbf{Proof of Theorem \ref{result 4: uniqueness Kantorovich solution with interaction}}]
To prove this theorem, we must show that $c_{\textbf{eff}}$ satisfies the assumptions of Brenier and Gangbo and McCann's theorems or Theorem 10.28 from \cite{villani2009optimal}. Remark \ref{pi independent of effective cost} tells that we can write the effective cost as $c_{\text{eff},\pi_{\text{min}}}$. This $\pi_{\text{min}}$ will help us attain $c_{\textbf{eff},\pi_{\text{min}}}$ implicitly depending on $\pi_{\text{min}}$. Pick $\pi_{\text{min}}$ by Theorem \ref{result 3: Kantorovich with interaction solutions}. Then $\pi_{\text{min}}$ is optimal with respect to $c_{\textbf{eff},\pi_{\text{min}}}$.  Since $\theta<\theta_{0}$, Lemma \ref{c-chypersurface} tells us $\nabla_{y}c_{\text{eff},\pi_{\text{min}}}$ is injective, and Lemma \ref{effective infimum} tells us $c_{\textbf{eff},\pi_{\text{min}}}$ is differentiable, and since $\mu_{0}\ll\;dx$, then $c_{\textbf{eff},\pi_{\text{min}}}$ satisfies the assumptions of Brenier and Gangbo and McCann's results. Therefore, there is a unique transport map $T$ solving $\textbf{Problem B}$.
\end{proof}

The proof of Theorem \ref{thm: B implies A} now follows as a corollary from Theorems \ref{result 1: Kantorovich on paths minimizers}, \ref{result 3: Kantorovich with interaction solutions} and, \ref{result 4: uniqueness Kantorovich solution with interaction}. Indeed,

\begin{proof}[\textbf{Proof of Theorem \ref{thm: B implies A}}]
From the previous results stated above, one has the following. If $\pi_{0}$ is a minimizer of $\mathcal{E}_{0}(\pi)$, then $\pi_{0}$ is a minimizer of $\mathcal{E}(\pi)$ without the interaction term containing $c_{0}(\gamma)$.
\end{proof}

\section{Acknowledgments}
My deepest gratitude goes to Nestor Guillen for bringing optimal transport to light! And for introducing this problem to me and for all his support and guidance throughout this research project. I would also like to thank Jun Kitagawa for reading a preliminary version of this note and making me aware of a standard construction in optimal transport known as dynamical couplings.
\appendix
\section{Bochner's Theorem}\label{Appendix}

A goal of this appendix is to demonstrate that the functional (\ref{interaction term}) with the Coulomb potential, $\kappa(x-y):=|x-y|^{2-n}$, is convex in $\pi$, which boils down in showing that the quadratic functional 

\begin{align*}
    \int_{\Omega}\int_{\Omega} \int_{0}^{1}|\gamma-\sigma|^{2-n}\;dt\;d\pi(\sigma)\;d\pi(\gamma)
\end{align*}($n\geq 3$) is convex in $\pi.$
In order to establish this we rely on Bochner's theorem as presented in Reed's and Simon's book \cite{reed1972methods}.  (see Theorem \ref{Bochner-Schwartz thm} below).

The generalized version of Bochner's theorem \cite[Theorem IX.9]{reed1972methods}, due to Schwartz \cite{schwartz1957theorie}, which includes distributions, is an extension of functions of positive type to distributions ( see definition \ref{distributions}).

\begin{theorem}\label{Bochner-Schwartz thm}\cite[Theorem IX.10]{reed1972methods} A distribution $T\in \mathcal{D}^{\prime}(\mathbb{R}^n)$  is a distribution of positive type if and only if $T\in \mathcal{S}^{\prime}(\mathbb{R}^n)$ and $T$ is the Fourier transform of a positive measure of at most  polynomial growth.
\end{theorem}

Some discussion and definitions are in order to fully understand this result. Suppose $f(x)$ is a bounded, continuous function. Then $f(x)$ is said to be of \emph{positive type} if 
\begin{align*}
    \int\int f(x-y)\overline{\varphi(y)}\varphi(x)\;dx\;dy\geq 0
\end{align*}
for all $\varphi \in C_{0}^{\infty}(\mathbb{R}^n)$, \cite{reed1972methods}. The condition above can be rewritten in the following way if we consider the convolution $(\overline{\Tilde{\varphi}}*\varphi)(\tau):=\int \overline{\Tilde{\varphi}(x-\tau)}\varphi(x)\;dx$, namely,
\begin{align*}
    \int\int f(\tau)\overline{\varphi(x-\tau)}\varphi(x)\;d\tau\;dx=\int f(\tau)(\overline{\Tilde{\varphi}}*\varphi)(\tau)\;d\tau\geq 0,
\end{align*}
where $\Tilde{\varphi}(x)=\varphi(-x).$ Then we have the definition:
\begin{definition}\label{distributions}
A distribution $T \in \mathcal{D}^{\prime}(\mathbb{R}^n)$ is said to be of \emph{positive type} if $T(\overline{\Tilde{\varphi}}*\varphi)\geq 0$ for all $\varphi \in \mathcal{D}(\mathbb{R}^n)$.
\end{definition}

We shall apply this to show the quadratic function given by a $\kappa$ of positive type,
\begin{align*}
    Q(\pi):=\int_{\Omega}\int_{\Omega}\int_{0}^{1}\kappa(\gamma(t)-\sigma(t))dt\;d\pi(\sigma)\;d\pi(\gamma)
\end{align*}
is convex in $\pi$. It suffices to show, for $\pi_{0}$ and $\pi_{1} \in \mathcal{P}(\Omega)$ with $\pi_{s}=(1-s)\pi_{0}+s\pi_{1}$, that $\frac{d^2}{ds^2}Q(\pi_{s})\geq 0$.  Indeed, 
\begin{align*}
    \frac{d^2}{ds^{2}}Q(\pi_{s})=2 \int_{0}^{1}\int_{\Omega}\int_{\Omega}\kappa(\gamma(t)-\sigma(t))d(\pi_{0}(\sigma)-\pi_{1}(\sigma))d(\pi_{0}(\gamma)-\pi_{1}(\gamma))dt.
\end{align*}
Let $\omega=\pi_{0}-\pi_{1}$. Let $d\omega_{t}=(e_{t})_{\sharp}\omega$. Then, up to a factor of $2$, this latter integral is equal to 
\begin{align*}
    \int_{0}^{1}\int_{\mathbb{R}^n}\int_{\mathbb{R}^n}\kappa(x-y)d\omega_{t}(x)d\omega_{t}(y)dt.
\end{align*}
Since the measures $\omega$ can be approximated under the weak-$*$ topology by finite measures against $C_{0}^{\infty}(\mathbb{R}^n)$ test functions, given finite measures $\rho\in \mathcal{M}(\mathbb{R}^n)$ with $\rho=\varphi(x)dx$ for all $\varphi\in C_{0}^{\infty}(\mathbb{R}^n)$, the above integral can be approximated by 
\begin{align*}
    \int_{0}^{1}\int_{\mathbb{R}^n}\int_{\mathbb{R}^n}\kappa(x-y)d\rho(x)d\rho(y)\;dt=\int_{0}^{1}\int_{\mathbb{R}^n}\int_{\mathbb{R}^n}\kappa(x-y)\varphi(x)\varphi(y)\;dx\;dy\;dt.
\end{align*}
Observe $\kappa(x)=C_{n}|x|^{2-n}$ has Fourier transform equal to $|\xi|^{-2}\geq 0$. From here one can see $\kappa$ is a distribution of the Fourier transform of a positive measure with polynomial growth, so by Bochner the above is $\geq 0$. This proves the statement at the start of this appendix.

It is worthwhile to mention an alternative proof using integration by parts. Let us consider the fundamental solution of the Laplace equation for $n\geq 3$: 
\begin{align*}
    \Phi(x)=C_{n}|x|^{2-n},\quad C_{n}\geq 0\;\text{a dimensional constant},
\end{align*}
(see Evans \cite[Ch 2]{evans1998partial}). Considering convolutions, for $f\in C_{c}^{2}(\mathbb{R}^n)$, we can write
\begin{align*}
    u(x)&=\int_{\mathbb{R}^n}\Phi(x-y)f(y)\;dy\\
    &=C_{n}\int_{\mathbb{R}^n}|x-y|^{2-n}f(y)\;dy.
\end{align*}
Thus \cite[Ch 2]{evans1998partial}, $u\in C_{c}^{2}(\mathbb{R}^n)$ and it solves the Poisson equation $-\Delta u=f$. 

Armed with this knowledge we readily show what we set to prove for the Coulomb potential, that the quadratic term in (\ref{interaction term}) is convex in $\pi$, that is $\frac{d^2}{ds^2}Q(\pi_{s})\geq 0$. More concretely, 
\begin{align*}
    \int_{\mathbb{R}^n}\int_{\mathbb{R}^n}|x-y|^{2-n}\varphi(x)\varphi(y)\;dx\;dy&=-\int_{\mathbb{R}^n}\Delta u_{\varphi}(x)\int_{\mathbb{R}^n}|x-y|^{2-n}\varphi(y)\;dy\;dx\\
    &=-C_{n}\int \Delta u_{\varphi}(x) u_{\varphi}(x)\;dx\\
    &=C_{n}\int |\nabla u_{\varphi}|^{2}\;dx\geq 0,
\end{align*}
where we applied integration by parts to get the last equation. 
\section{The differentiability of the end point cost function}\label{Appendix2}

In this appendix we  prove the following limit exists:
\begin{equation}\label{limit quotient}
    \lim_{h \to 0}\frac{c_{\textbf{e}}(x, y+h\hat{e})-c_{\textbf{e}}(x, y)}{h}.
\end{equation}
This stems from observing if $\gamma_{x,y}(t)$ is the minimal path from $x:=\gamma_{x,y}(0)$ to $y:=\gamma_{x,y}(1)$, $c_{\textbf{e}}(x,y)=c(\gamma_{x,y}(t))$ and, if we momentarily suppose that the above limit exists, then said limit will equal the left-hand side of the following equation:
\begin{align*}
    \nabla_{y}c_{\textbf{e}}(x,y)=\frac{\partial c}{\partial y}\nabla_{y}\gamma_{x,y},
\end{align*}
by the chain rule. Indeed, 
\begin{align*}
    \nabla_{y}\left(c_{\textbf{e}}(x,y)\right)=\int_{0}^{1}\dot{\gamma}_{x,y}(t)\cdot\nabla_{y}\left(\dot{\gamma}_{x,y}(t)\right)\nabla_{y}\gamma_{x,y}(t)-\nabla V(\gamma_{x,y}(t),t)\nabla_{y}\gamma_{x,y}(t)\;dt,
\end{align*}
and since
\begin{align*}
    \nabla_{y}\left(\dot{\gamma}_{x,y}(t)\right)=\nabla_{y}\left(\frac{d}{dt}\gamma_{x,y}(t)\right)=\frac{d}{dt}\;\nabla_{y}\left(\gamma_{x,y}(t)\right),
\end{align*}
to show the limit (\ref{limit quotient}) exists it suffices to show the following limit exists:
\begin{align*}
    \lim_{h\to 0}\frac{\gamma_{x,y+h\;\hat{e}}(t)-\gamma_{x,y}(t)}{h}.
\end{align*}

\begin{lemma}\label{end cost function diff}
The end point cost function $c_{\textbf{e}}(x,y):=\inf_{\gamma(0)=x,\;\gamma(1)=y}c(\gamma)$ is differentiable with respect to $y$. Moreover,
\begin{align*}
   \nabla_{y}c_{\textbf{e}}(x,y)=\lim_{h \to 0}\frac{c_{\textbf{e}}(x, y+h\hat{e})-c_{\textbf{e}}(x, y)}{h}.
\end{align*}

\end{lemma}
\begin{proof}
For all $t\in [0, 1]$, let $\theta_{h}: [0, 1]\to X$ be defined by
\begin{align*}
\theta_{h}(t):=\frac{\gamma_{x,y+h\;\hat{e}}(t)-\gamma_{x,y}(t)}{h},\quad (h\neq 0).
\end{align*}
From the aforementioned, to prove the lemma it suffices to prove $\lim_{h \to0}\theta_{h}(t)$ exists. This will be accomplished by applying a ``standard"\footnote{This is commonly used, for instance, in the context of viscosity solutions.} argument of compactness plus uniqueness of linear ODEs. 

Thanks to Proposition \ref{minimal paths plus potential}, $\gamma_{x, y+h\hat{e}}(t)$ solves
\begin{align*}
    \ddot{\gamma}_{x,y+h\hat{e}}(t)=-\nabla V(\gamma_{x,y+h\hat{e}}(t),t).
\end{align*}
Then $\theta_{h}(t)$ solves the boundary value problem:
\begin{align*}
     \ddot{\theta}_{h}(t)&=-\frac{1}{h}\left(\nabla V(\gamma_{x,y+h\hat{e}}(t), t)-\nabla V(\gamma_{x,y}(t),t)\right),\\
     \theta_{h}(0)&=0, \\
     \theta_{h}(1)&=\hat{e}.
     \end{align*} 
Assuming $V\in C^{2}(X)$, Taylor expanding about $h$, we get
\begin{align}\label{a taylor expansion}
     \ddot{\theta}_{h}(t)&=-D^{2}V(\gamma_{x,y}(t),t)\theta_{h}(t)+o_{h}(1)(t),\\
     \theta_{h}(0)&=0, \\
     \theta_{h}(1)&=\hat{e}.
\end{align}

Proposition \ref{boundary value problem} tells us that $\theta_{h}(t)$ is equibounded:
\begin{align*}
    \|\theta_{h}(t)\|_{\infty}\leq \frac{1}{1-L}=:C_{0}\quad \forall\; h, t.
\end{align*}
Since $\nabla V$ is $L$-Lipschitz, and another application of Proposition \ref{boundary value problem}, we have 
\begin{align*}
    \left\|\ddot{\theta}_{h}(t)\right\|_{\infty}\leq \frac{L}{1-L}=:C_{1}\quad\forall\; h, t,
\end{align*}
and so $\ddot{\theta}_{h}$ is equibounded as well. This implies $\theta_{h}(t)$ is equicontinuous. For the reader's benefit we show how this follows. More concretely, for any interval $[p, q]$ such that $0\leq p<q\leq 1$ with $q-p=1$, the Mean Value Theorem applies to show there is some $\tau \in [p, q]$ such that 
\begin{align*}
    \theta_{h}(q)-\theta_{h}(p)=\dot{\theta}_{h}(\tau).
\end{align*}
Then $\left\|\dot{\theta}_{h}(\tau)\right\|_{\infty}\leq 2\;C_{0}$.
Another application of the Mean Value Theorem shows there is some $\xi \in [r, s]$ such that
\begin{align*}
    \frac{\dot{\theta}_{h}(s)-\dot{\theta}_{h}(r)}{s-r}=\ddot{\theta}_{h}(\xi)\quad (0\leq p\leq r<s\leq q\leq 1).
\end{align*}
Using this, pick any $\eta\in [p, q]$ so that
\begin{align*}
    \left\|\dot{\theta}_{h}(\eta)-\dot{\theta}_{h}(\tau)\right\|\leq C_{1},
\end{align*}
while the reverse triangle inequality gives us
\begin{align*}
    C_{1}\geq\left\|\dot{\theta}_{h}(\eta)-\dot{\theta}_{h}(\tau)\right\|\geq\left\|\dot{\theta}_{h}(\eta)\right\|-\left\|\dot{\theta}_{h}(\tau)\right\|
\end{align*}
which implies
\begin{align*}
    \left\|\dot{\theta}_{h}(\eta)\right\|\leq 2\;C_{0}+C_{1}:=C_{2}\quad \forall\; \eta\in [p, q].
\end{align*}

For any $0\leq p \leq r<x<y<s\leq q\leq 1$, once again the Mean Value Theorem applies to show 
\begin{align*}
    \left\|\theta_{h}(y)-\theta_{h}(x)\right\|\leq C_{2}|y-x|,
\end{align*}
and hence $\theta_{h}$ is equicontinuous. We are ready to prove the existence of the limit in $\gamma$. For any sequence $h_{k}$, such that $h_{k}\to 0$ as $k\to \infty$, the family of functions $\{\theta_{h_{k}}(t)\}$ from $[0, 1]$ to $X$ are both equibounded and equicontinuous. Arzelá-Ascoli applies to show that the sequence $\{\theta_{h_{k}}\}$ admits a subsequence $\{\theta_{\widetilde{h}_{k}}\}$ uniformly converging, as $k\to \infty$, to a continuous function $\theta:[0, 1]\to X$. From here, using (\ref{a taylor expansion}), one can show $\theta$ is $C^{2}([0, 1])$ and uniquely solves 
\begin{align*}
     \ddot{\theta}(t)&=-D^{2}V(\gamma_{x,y}(t),t)\theta(t)\\
     \theta(0)&=0, \\
     \dot{\theta}(0)&=\hat{e}.
     \end{align*} 
\end{proof}
\bibliography{OTrefs}
\bibliographystyle{plain}

(Rene Cabrera) Department of Mathematics and Statistics, University of Massachusetts, Amherst, MA  01003-9305\\

\end{document}